\documentclass[12pt]{amsart}
\usepackage{amsthm,amsmath,amsfonts,amssymb,latexsym,amscd,mathrsfs,hyperref,graphicx}
\usepackage{setspace}
\usepackage{cite}
\usepackage{phonetic}
\setdisplayskipstretch{2}
\usepackage{cite}

\makeatletter
\renewcommand{\pod}[1]{\allowbreak\mathchoice
  {\if@display \mkern 18mu\else \mkern 8mu\fi (#1)}
  {\if@display \mkern 18mu\else \mkern 8mu\fi (#1)}
  {\mkern4mu(#1)}
  {\mkern4mu(#1)}
}

\makeatletter
\@namedef{subjclassname@2020}{%
  \textup{2020} Mathematics Subject Classification}
\makeatother

\theoremstyle{plain}
\newtheorem{thm}{Theorem}
\newtheorem{lem}{Lemma}

\newtheorem*{thm*}{Theorem}
\newtheorem*{cor*}{Corollary}
\newtheorem*{prop*}{Proposition}

\theoremstyle{definition}

\newtheorem*{exa*}{Example}

\newcommand{\ds}{\displaystyle}
\newcommand{\bs}{\boldsymbol}
\newcommand{\mb}{\mathbb}
\newcommand{\mc}{\mathcal}

\def \a{\alpha} \def \b{\beta}   \def \g{\gamma}  \def \l{\lambda}  \def \t{\theta} 

\makeatletter
\def\widebreve{\mathpalette\wide@breve}
\def\wide@breve#1#2{\sbox\z@{$#1#2$}%
     \mathop{\vbox{\m@th\ialign{##\crcr
\kern0.08em\brevefill#1{0.8\wd\z@}\crcr\noalign{\nointerlineskip}%
                    $\hss#1#2\hss$\crcr}}}\limits}
\def\brevefill#1#2{$\m@th\sbox\tw@{$#1($}%
  \hss\resizebox{#2}{\wd\tw@}{\rotatebox[origin=c]{90}{\upshape(}}\hss$}
\makeatletter

\numberwithin{equation}{section}

\renewcommand{\labelenumi}{\setlength{\labelwidth}{\leftmargin}
   \addtolength{\labelwidth}{-\labelsep}
   \hbox to \labelwidth{\theenumi.\hfill}}

\begin{document}
\title{Some Diophantine equations and inequalities with primes}
\author{Roger Baker}
\address{Department of Mathematics\newline
\indent Brigham Young University\newline
\indent Provo, UT 84602, U.S.A}
\email{baker@math.byu.edu}

 \begin{abstract}
We consider the solutions to the inequality
 \[|p_1^c + \cdots + p_s^c - R| < R^{-\eta}\]
(where $c > 1$, $c \not\in \mb N$ and $\eta$ is a small positive number; $R$ is large). We obtain new ranges of $c$ for which this has many solutions in primes $p_1, \ldots, p_s$, for $s = 2$ (and `almost all' $R$), $s=3$, 4 and 5.

We also consider the solutions to the equation in integer parts
 \[[p_1^c] + \cdots + [p_s^c] = r\]
where $r$ is large. Again $c> 1$, $c\not\in \mb N$. We obtain new ranges of $c$ for which this has many solutions in primes, for $s=3$ and 5.
 \end{abstract}



\maketitle

 \section{Introduction}\label{sec:intro}
 
Let $c > 1$, $c \not\in \mb N$. Let $\eta$ be a small positive number depending on $c$. Let $R$ be a large positive number. We consider solutions in primes of the inequality
 \[|p_1^c + \cdots + p_s^c - R| < R^{-\eta}\tag*{$(1)_s$}
 \label{eq(1)s}\]
first studied by \u Sapiro-Pyateck\(\widebreve{\i\i}\) \cite{sap}. We also consider the equation in integer parts
 \[[p_1^c] + \cdots + [p_s^c]=r.\tag*{$(2)_s$}
 \label{eq(2)s}\]
We give results providing large numbers of solutions of \ref{eq(1)s} $(s=2,3,4,5)$ and \ref{eq(2)s} $(s=3,5)$ for new ranges of $c$. For $s=2$, one needs to restrict $R$ to `almost all' real numbers in an interval $[V,2V]$.

Following a nice innovation in a paper of Cai \cite{cai3}, there has been recent progress in all these cases; see below for details. In the present paper progress is made by combining this innovation with the powerful exponential sum bounds of Huxley \cite{hux}, Bourgain \cite{bou}, and Heath-Brown \cite{hb3}. When discussing $(1)_3$, $(1)_5$ and $(2)_5$, we use a vector sieve in conjunction with the Harman sieve. The other cases are simpler and Heath-Brown's generalized Vaughan identity replaces the sieve method.

We write `$n\sim N$' to signify $N < n \le 2N$. Let $X = R^{1/c}$.

Let $\mc A_s(R)$ denote the number of solutions of \ref{eq(1)s} with $\frac X8 < p_j \le X$ $(j=1,\ldots, s)$. Let $\mc B_s(r)$ denote the number of solutions of \ref{eq(2)s} with $\frac X8 < p_j \le X$ $(j = 1, \ldots,s)$. One expects heuristically to obtain (at least for $c$ not too large) the bounds
 \[\mc A_s(R) \gg \frac{R^{\frac sc - 1 - \eta}}
 {(\log R)^s}\tag*{$(3)_s$}\label{eq3s}\]
and
 \[\mc B_s(r) \gg \frac{r^{\frac sc-1}}{(\log r)^s}.\
 \tag*{$(4)_s$}\label{eq4s}\]

 \begin{thm}\label{thm1}
Let $V$ be large. Suppose that $c < \frac{39}{29} = 1.3448\ldots$, $c \ne 4/3$. We have $(3)_2$ for all $R$ in $[V,2V]$ except for a set of $R$ having measure $O(V \exp(-C(\log V)^{1/4}))$.
 \end{thm}

(We denote by $C$ a positive absolute constant, not the same at each occurrence.)

Previous upper bounds for permissible $c$:
 \bigskip
 
 \centerline{17/16 \cite{lap},\ 15/14 \cite{lap2},\  
 43/36 \cite{caozhai4},\ 59/44 = 1.3409\ldots \cite{caili}.}
 \bigskip
 
 \begin{thm}\label{thm2}
Let $R$ be large. Suppose that $c < 6/5$. Then $(3)_3$ holds. 
 \end{thm}

Previous upper bounds:
 \bigskip
 
 \centerline{15/14 \cite{lap}, \ 13/12 \cite{cai}, \
 11/10 \cite{cai2, kumned}, \ 237/214 \cite{caozhai},}
  \bigskip
  
 \centerline{$\dfrac{61}{55}$ \cite{kum}, \ $\dfrac{10}9$ \cite{bakwein2},
 \ $\dfrac{43}{36} = 1.1944\ldots$ \cite{cai3}.}
 \bigskip
 
 \begin{thm}\label{thm3}
Let $R$ be large. Suppose that $c < 39/29$. Then $(3)_4$ holds.
 \end{thm}

Previous upper bounds:
 \bigskip

\centerline{$\dfrac{97}{81}$ \cite{liushi}, \ $\dfrac 65$ \cite{lizha3}, \ $\dfrac{59}{44}$ \cite{caili}, \ $\dfrac{1198}{889} = 1.3419\ldots$ \cite{lizha5}.}
 \bigskip

 \begin{thm}\label{thm4}
Let $R$ be large. Suppose that $c < \frac{378}{181} = 2.0883\ldots$. Then $(3)_5$ holds.
 \end{thm}

Previous upper bounds:
 \bigskip

\centerline{$1.584\ldots$ \cite{caozhai2}, \ $\dfrac{1+\sqrt 5}2$ \cite{gar},\ $\dfrac{81}{40}$ \cite{caozhai3},\ $\dfrac{108}{53}$ \cite{liushi}, 2.041 \cite{bakwein},}
 \bigskip

\centerline{2.08 \cite{caili}, \ $\dfrac{665576}{319965} = 2.0801\ldots$ \cite{lizha2}.}
 \bigskip

 \begin{thm}\label{thm5}
Let $n$ be large. Suppose that $c < \frac{3581}{3106} = 1.1529\ldots$. Then $(4)_3$ holds.
 \end{thm}

Previous upper bounds:
 \bigskip

\centerline{$\dfrac{17}{16}$ \cite{laptol}, $\dfrac{12}{11}$ \cite{kumned}, $\dfrac{258}{235}$ \cite{caozhai}, $\dfrac{137}{119}$ \cite{cai4}, $\dfrac{3113}{2703} = 1.1516\ldots$ \cite{lizha4}.}
 \bigskip

 \begin{thm}\label{thm6}
Let $n$ be large. Suppose that $c < \frac{609}{293} = 2.0784\ldots$. Then $(4)_5$ holds.
 \end{thm}

Previous upper bounds:
 \bigskip

\centerline{$\dfrac{4109054}{1999527}$ \cite{lizha}, $\dfrac{408}{197} = 2.071\ldots$ \cite{li}.}
 \bigskip

Along usual lines, we employ a continuous function $\phi : \mb R \to [0,1]$ such that
 \begin{equation}\label{eq1.1}
\phi(y) = 0 \quad (|y| \le R^{-\eta}), \ \phi(y) = 1 \quad \left(|y| \le \frac{4R^{-\eta}}5\right),
 \end{equation}
with Fourier transform
 \[\Phi(x) := \int_{-\infty}^\infty e(-xy) \phi(y)dy \ \ ,
 \ \text{where } e(\t) := e^{2\pi i\t},\]
satisfying
 \begin{equation}\label{eq1.2}
\int_{|x|>X^{2\eta}} |\Phi(x)|dx \ \ \ll X^{-3}.
 \end{equation}
We define
 \[\tau = X^{8\eta - c}, K = X^{2\eta}, \mc L = \log X,
 P(z) = \prod_{p<z} p.\]
Let $\rho(n)$ denote the indicator function of the prime numbers. For $u \in \mb N$, $z > 1$, let
 \[\rho(u,z) = 1 \ \text{ if } (u, P(z)) = 1.\]
Let $\rho(u,z)=0$ otherwise. For a vector sieve one usually uses functions $\rho^-(\ldots)$, $\rho^+(\ldots)$ with
 \[\rho^-(n) \le \rho(n) \le \rho^+(n);\]
but (without loss) we shall take $\rho^- = \rho$, so that the inequality basic to \cite{brufou} becomes
 \begin{equation}\label{eq1.3}
\rho(m) \rho(\ell) \ge \rho^+(m) \rho(\ell) + \rho(m) \rho^+(\ell) - \rho^+(m) \rho^+(\ell). 
 \end{equation}
We shall need exponential sums
 \begin{align*}
S(x) &= \sum_{\frac X8 < n\le X} \rho(n) e(n^cx) , S_1(x) = \sum_{\frac X8 <, n\le X} \rho(n) \log n\ e(n^cx),\\[2mm]
&\qquad S^+(x) =\sum_{\frac X8 < n\le X} \rho^+(n) e(n^cx),\\[2mm]
T(x) &= \sum_{\frac X8 < n\le X} \rho(n) e([n^c]x), T_1(x) = \sum_{\frac X8 < n\le X} \rho(n) \log n\ e([n^c]x),\\[2mm]
&\qquad T^+(x) = \sum_{\frac X8 < n\le X} \rho^+(n) e([n^c]x),\\[2mm]
\intertext{and approximating functions}
I(x) &= \int_{X/8}^X e(t^c x) dt,\\[2mm]
J(x) &= \sum_{\left(\frac X8\right)^c < m \le X^c} \ \frac 1c\, m^{\frac 1c -1} e(xm).
 \end{align*}
In using Cai's idea we also need the sums
 \[A(x) = \sum_{\frac X8 < n\le X} e(n^cx) \ , \ B(x) =
 \sum_{\frac X8 < n\le X} e([n^c]x).\]

We now describe briefly the underlying principle of the proofs. For $(3)_2$, $(3)_4$ we use
 \begin{align}
\mc L^s \mc A_s(R) &\gg \sum_{\frac X8 < p_j\le X\, (j=1,\ldots,s)} \log p_1\ldots \log p_s \phi(p_1^c + \cdots + p_s^c - R)\label{eq1.4}\\[2mm]
&= \sum_{\frac X8 < p_j\le X (j=1,\ldots,s)}\log p_1\ldots \log p_s \int_{\mb R} \Phi(x) e((p_1^c + \cdots + p_s^c)x)e(-Rx)dx\notag\\[2mm]
&= \int_{\mb R} S_1(x)^s \Phi(x) e(-Rx)dx.\notag
 \end{align}
For $(4)_3$ we use
 \begin{align}
\mc L^s\mc B_s(r) &\gg \sum_{\frac X8 < p_j\le X\, (j=1,2,3)}\log p_1\ldots \log p_s \int_{-1/2}^{1/2} e(([p_1^c] + [p_2^c] + [p_3^c] - r)x)dx\label{eq1.5}\\[2mm]
&= \int_{-1/2}^{1/2} T_1(x)^3 e(-rx)dx.\notag
 \end{align}
Modifying this for $(3)_3$, $(3)_5$, we use
 \begin{align}
\mc A_s(R) &\ge \sum_{\frac X8 < m_1,\ldots, m_{s-2}, m,\ell \le X} \rho(m_1)\ldots \rho(m_{s-2}).\label{eq1.6}\\[2mm]
&\hskip .75in (\rho^+(m)\rho(\ell) + \rho(m)\rho^+(\ell) - \rho^+(m) \rho^+(\ell)).\notag\\[2mm]
&\hskip .75in \phi(m_1^c + \cdots + m_{s-2}^c + m^c + \ell^c - R)\notag\\[2mm]
&= \int_{-\infty}^\infty S(x)^{m-2}(2S(x)S^+(x) - S^+(x)^2) \Phi(x) e(-Rx)dx\notag
 \end{align}
and for $(4)_5$,

 \begin{align}
\mc B_s(n) &\ge \sum_{\frac X8 < m_1, m_2, m_3, m, n\le X}\rho(m_1) \rho(m_2) \rho(m_3).\label{eq1.7}\\[2mm]
&\hskip .75in (\rho^+(m)\rho(n) + \rho(m)\rho^+(n) - \rho^+(m)\rho^+(\ell)).\notag\\[2mm]
&\hskip .75in \int_{-1/2}^{1/2} e(([m_1^c] + [m_2^c] + [m_3^c] + [m^c] + [\ell^c] - r)x)dx\notag\\[2mm]
&= \int_{-1/2}^{1/2} (2T(x)^4 T^+(x) - T(x)^3 T^+(x)^2) e(-rx)dx.\notag
 \end{align}
Note that the function $\rho^+$ in Theorem \ref{thm3} is different (although we still write $\rho^+$) from the function in Theorems \ref{thm4} and \ref{thm6}.

We give a few more indications of method in the simplest case $(3)_4$. The `major arc' $\mc M$ is $(-\tau,\tau)$, and $\mb R\backslash\mc M$ is the `minor arc'.
 \medskip

(i) Show that the last integral in \eqref{eq1.4} reduces to the corresponding integral over $(-\tau,\tau)$ with acceptable error (the `minor arc' stage).

(ii) Show that the integrand over $(-\tau,\tau)$ can be replaced by $\left(\int_X^{2X}\, \frac{e(t^cx)}{\log t}\, dt\right)^3$. $\Phi(x) e(-Rx)$ with acceptable error (first part of `major arc' stage).

(iii) Extend the integral from $(-\tau,\tau)$ to $\mb R$ with acceptable error and obtain the lower bound $(3)_3$ for this last integral (second part of `major arc' stage).

The other cases \ref{eq3s}, \ref{eq4s} are similar in principle, but more complicated. The innovations are in part (i) in each case.

I would like to thank Andreas Weingartner for computer calculations involving integrals and exponent pairs.
\bigskip

 \section{Lemmas for the minor arc.}\label{sec:lemminorarc}

For $N \ge 1$, we write $I(N)$ for a subinterval of $(N, 2N]$, not the same at each occurrence. For a real function $f$ on $[N, 2N]$ we write
 \[S(f,N) = \sum_{n\,\in\,I(N)} e(f(n)).\]
The fractional part of $x$ is written as $\{x\}$. We write $A\asymp B$ for $A\ll B\ll A$. Implied constants in the conclusions of the lemmas depend on $c$, $\eta$ if these appear, together with any implied constants in the hypotheses.

 \begin{lem}\label{lem1}
Let $x > 0$. For real numbers $a_n$, $|a_n| \le 1$, let
 \[W(X, x) = \sum_{\frac X8 < n\le X} a_ne(x[n^c]).\]
For $2 \le H \le X$, we have
 \begin{align*}
W(X,x) &\ll \frac{X\mc L}H + \sum_{0 \le h \le H}\min\left(1, \frac 1h\right)\Bigg|\sum_{\frac X8 < n\le X} a_ne ((h+\g)m^c)\Bigg|\\[2mm]
&+ \sum_{1 \le h \le H} \frac 1h\ \Bigg|\sum_{\frac X8 < n\le X} e(hn^c)\Bigg| + \sum_{h > H}\, \frac H{h^2}\ \Bigg|\sum_{\frac X8 < n\le X} e(hn^c)\Bigg|.
 \end{align*}
Here $\g = \{x\}$ or $-\{x\}$.
 \end{lem}
 
 \begin{proof}
Let $\a = \{x\}$. By \cite{rcb2}, Lemma 2.3, we have
 \[e(-\a\{t\}) = \sum_{|h|\le H} c_h(\a) e(ht)
 + O\left(\min\left(1, \frac 1{H\|t\|}\right)\right)\]
$(t\in \mb R)$, where
 \[c_h(\a) = \frac{1-e(-\a)}{2\pi i(h+\a)}.\]
Moreover,
 \[\min\left(1, \frac 1{H\|t\|}\right) = 
 \sum_{h=-\infty}^{\infty} c_h e(ht)\]
where
 \[c_h \ll \min\left(\frac{\log H}H, \frac 1{|h|},
 \frac H{h^2}\right);\]
see e.g.~\cite{hb2}. Note that for real $u$, this implies
 \begin{align*}
e(\a[u]) &= e(\a u)e(-\a\{u\})\\
&= e(\a u) \sum_{|h| \le H} c_h(\a)\, e(hu) + O\Bigg(
\sum_{h=-\infty}^\infty c_h e(hu)\Bigg).
 \end{align*}
Set $u = n^c$, so that
 \[e(x[n^c]) = e(\a[n^c]).\]
Summing over $n$,
 \begin{align*}
W(X,x) = \sum_{\frac X8 < n\le X} a_ne(\a[n^c]) &= \sum_{|h|\le H}
 c_h(\a) \sum_{\frac X8 < n\le X} a_ne((h+\a)n^c)\\[2mm]
&\qquad + O\Bigg(\sum_{h=-\infty}^\infty c_h \, \sum_{\frac X8 < n\le X} e(hn^c)\Bigg).
 \end{align*}
The `$O$' term yields
 \[\ll \frac{X\log X}H + \sum_{1 \le h \le H} \frac 1h\
 \Bigg|\sum_{\frac X8 < n\le X} e(hn^c)\Bigg| + \sum_{h > H} 
 \frac H{h^2} \ \Bigg|\sum_{\frac X8 < n\le X} e(hn^c)\Bigg|.\]
For the remaining terms we use
 \[c_h(\a) \ll \begin{cases}
 1 & \text{if $h=0$}\\[2mm]
 \dfrac 1h & (h\ne 0) \end{cases}\]
and note that for $h = -1, -2, \ldots, -[H]$ we have
 \[h+\a = -(|h|-\a).\qedhere\]
 \end{proof}
 
 \begin{lem}[Kusmin-Landau]\label{lem2}
If $f$ is continuously differentiable, $f'$ is monotonic, and $|f'| \ge \l$ on $[N, 2N]$, then
 \[S(f,N) \ll \l^{-1}.\]
 \end{lem}
 
 \begin{proof}
\cite[Theorem 2.1]{grakol}.
 \end{proof}
 
 \begin{lem}[$A$ process]\label{lem3}
For $1 \le Q \le N$,
 \[S(f, N)^2 \ll \frac NQ \ \sum_{|q| < Q} \
 |S(f_q,N)|\]
where $f_q(x) = f(x+q) - f(x)$.
 \end{lem}

 \begin{proof}
\cite[p. 10]{grakol}.
 \end{proof}

 \begin{lem}[$B$ process]\label{lem4}
Suppose that $f'' \asymp FN^{-2}$ on $[N,2N]$, where $F > 0$, and
 \[f^{(j)}(x) \ll FN^{-j} \quad (j = 3, 4) 
 \text{ for } x \in [N,2N].\]
Define $x_\nu$ by $f'(x_\nu) = \nu$ and let
 \[\phi(\nu) = -f(x_\nu) + \nu x_\nu.\]
Then for $I(N) = [a,b]$,
 \[S(f,N) = \sum_{f'(b) \le \nu \le f'(a)} 
 \frac{e(-\phi(\nu)-1/2)}{|f''(x_\nu)|^{1/2}} 
 + O(\log(FN^{-1} + 2) + F^{-1/2}N).\]
 \end{lem}
 
 \begin{proof}
\cite[Lemma 3.6]{grakol}.
 \end{proof}
 
 \begin{lem}\label{lem5}
(i) Let $\ell \ge 0$ be a given integer, $L = 2^\ell$. Suppose that $f$ has $\ell + 2$ continuous derivatives on $[N,2N]$ and
 \[|f^{(r)}(x)| \asymp FN^{-r}
 (r = 1, \ldots, \ell + 2, x \in[N,2N]).\]
Then
 \[S(f, N) \ll F^{1/(4L-2)}N^{1-(\ell + 2)/(4L-2)}
 + F^{-1}N.\]
(ii) Let $f(x) = yx^b$ where $y \ne 0$, $\frac b{b-1} \not\in \{2, 3, \ldots, \ell + 1\}$. With $F = |y|N^c$ and $\ell$, $L$ as in (i),
 \[S(f, N) \ll F^{1/2-\frac{\ell+1}{4L-2}} \,
 N^{\frac{\ell+2}{4L-2}} + NF^{-1}.\]
 \end{lem}
 
 \begin{proof}
Part (i) is Theorem 29 of \cite{grakol}. If $FN^{-1} < \eta$, part (ii) follows from the Kusmin-Landau theorem. Suppose now that $FN^{-1} \ge \eta$. We apply the $B$ process to $S(f,N)$ and then part (i) of the lemma to the resulting sum (after a partial summation). This yields
 \begin{align*}
S(f,N) &\ll NF^{-1/2}\left(F^{1/(4L-2)}(FN^{-1})^{1-\frac{\ell+2}{4L-2}} + (FN^{-1})F^{-1}\right)\\[2mm]
&\hskip 1in + NF^{-1/2} + 1.
 \end{align*}
The last three terms can be absorbed into the first term.
 \end{proof}
 
 \begin{lem}\label{lem6}
Suppose that $g^{(6)}$ is continuous on $[1,2]$ and
 \begin{alignat*}{2}
g^{(j)}(x) &\ll 1 \quad &&(5 \le j \le 6)\\[2mm]
g^{(j)}(x) &\asymp 1 &&(2 \le j \le 4).
 \end{alignat*}
Let $T$, $N$ be positive with
 \[T^{1/3} \ll N \ll T^{1/2}.\]
Let
 \[S_h = \sum_{m\in I_h} e\left(Ty_hg\left(
 \frac mN\right)\right)\]
where $I_h$ is a subinterval of $[N,2N]$ and $y_1,\ldots, y_H \in[1,2]$ with
 \[y_{j+1} - y_j \gg \frac 1H \quad (j=1,\ldots, H-1).\]
Then
 \begin{align*}
\sum_{h=1}^H |S_h| &\ll H^{319/345} N^{449/690} T^{\frac{63}{690}+\eta}\\[2mm]
&\qquad + HN^{1/2} T^{141/950 + \eta}.
 \end{align*}
 \end{lem}
 
 \begin{proof}
We combine a special case of \cite[Theorem 2]{hux} with H\"older's inequality.
 \end{proof}
 
 \begin{lem}\label{lem7}
Let $g$ be a function with derivatives of all orders on $\left[\frac 12, 1\right]$ and
 \[|g^{(j)}(x)| \gg 1 \quad \left(x \in 
 \left[\frac 12, 1\right], 2 \le j \le 4\right).\]
Let $T$, $N$ be positive with
 \[T^{17/42} \le N \le T^{25/42}.\]
Then
 \[\sum_{n\in I(N)} e\left(Tg\left(\frac nN\right)\right)
 \ll N^{1/2} T^{\frac{13}{84}+\eta}.\]
 \end{lem}
 
 \begin{proof}
Theorem 4 of \cite{bou} is the case
 \[T^{17/42} \le N \le T^{\frac 12} \ , \
 I(N) = [N, 2N].\]
On pages 222--223 of \cite{bou} it is indicated how to extend this to $\left[T^{\frac 12}, T^{25/42}\right]$ using the $B$ process. An application of \cite[Lemma 7.3]{grakol} enables one to replace $[N,2N]$ by $I(N)$ with the loss of a log factor.
 \end{proof}
 
 \begin{lem}\label{lem8}
Let $k\in \mb N$, $k \ge 3$. Let $f$ have continuous derivatives $f^{(j)}$ $(1 \le j \le k)$ on $[0,N]$,
 \[|f^{(k)}(x)|\asymp \l_k\ \text{ on }\ (0,N].\]
Then
 \begin{equation}\label{eq2.1}
S(f,N) \ll N^{1+\eta}\Big(\l_k^{\frac 1{k(k-1)}} + N^{-\frac 1{k(k-1)}} + N^{-\frac 2{k(k-1)}} \l_k^{-\frac 2{k^2(k-1)}}\Big).
 \end{equation}
 \end{lem}
 
 \begin{proof}
\cite[Theorem 1]{hb3}.
 \end{proof}
 
 \begin{lem}\label{lem9}
Let $\t$, $\phi$ be real constants,
 \[\t(\t-1)(\t-2)\phi(\phi-1)(\t + \phi-2)(\t + \phi-3)
 (\t + 2\phi-3)(2\t + \phi-4) \ne 0.\]
Let $F\ge 1$ and let $|a_m| \le 1$. Let
 \[T(\t, \g) = \sum_{m\sim M} a_m \sum_{n\in I_m}
 e\left(\frac{Fm^\t n^\phi}{M^\t N^\phi}\right)\]
where $I_m$ is a subinterval of $(N, 2N]$. Then
 \begin{align*}
T(\t, \g) \ll (MN)^\eta &(F^{3/14} M^{41/56}
 N^{29/56} + F^{1/5} M^{3/4} N^{11/20}\\[2mm]
&+ F^{1/8} M^{13/16} N^{11/16} + M^{3/4} N + MN^{3/4}
+ MNF^{-1}).
 \end{align*}
 \end{lem}

 \begin{proof}
\cite[Theorem 2]{bakwein2}.
 \end{proof}
 
We write $(\a)_0 = 1$, $(\a)_s = (\a)_{s-1}(\a + s - 1)$ $(s=1,2,\ldots)$.

 \begin{lem}\label{lem10}
Let $\t$, $\phi$ be real
 \[(\t)_4 (\phi)_4 (\t + \phi + 2)_2 \ne 0.\]
Let $MN \asymp X$, $F \ge 1$, $|a_m| \le 1$. Let $N_0 = \min(M, N)$. Then in the notation of Lemma \ref{lem9},
 \begin{align*}
T(-\t,-\phi) \ll X^\eta (X^{11/12} &+ XN^{-1/2} + F^{1/8} X^{13/16} N^{-1/8}\\[2mm]
&+ (FX^5 N^{-1}N_0^{-1})^{1/6} + XF^{-1}).
 \end{align*}
 \end{lem}

 \begin{proof}
\cite[Lemma 9]{bakwein}.
 \end{proof}
 
 \begin{lem}\label{lem11}
Let $|a_m| \le 1$, $|b_n| \le 1$. Let
 \[S = \sum_{m\sim M} a_m \sum_{n\sim N} b_ne
 (Bm^\b n^\a)\]
where $M \ge 1$, $N\ge 1$, $\a(\a-1)(\a-2)\b(\b-1)(\b-2)\ne 0$. Suppose that
 \[F := BM^\b N^\a \gg X.\]
Then
 \begin{align}
SX^{-\eta} &\ll F^{1/20} N^{19/20} M^{29/40} + F^{3/46} N^{43/46} M^{16/23}\label{eq2.2}\\[2mm]
&+ F^{1/10} N^{9/10} M^{3/5} + F^{3/28} N^{23/28} M^{41/56}\notag\\[2mm]
&+ F^{1/11} N^{53/66}M^{17/22}+ F^{2/21} N^{31/42} M^{17/21}\notag\\[2mm]
&+ F^{1/5} N^{7/10} M^{3/5} + N^{1/2}M + F^{1/8} (NM)^{3/4}.\notag
 \end{align}
 \end{lem}
 
 \begin{proof}
This is due to Sargos and Wu \cite{sarwu}. Full details are given in \cite[proof of Theorem 3]{bakwein2}.
 \end{proof}
 
 \begin{lem}\label{lem12}
Let $0 < B < K$ and $|c_n| \le 1$. Let
 \[V(x) = \sum_{\frac X8 < n\le X} c_ne(n^cx) \ \text{ or } \
 \sum_{\frac X8 < n\le X} c_ne([n^c]x).\]
Then
 \begin{align*}
\ds \int_B^{2B} |V(y)|^2dy &\ll XB + X^{2-c}\mc L,\tag{i}\\[2mm]
\ds \int_B^{2B} |V(y)|^4dy &\ll (X^2B + X^{4-c}) X^\eta \ (c > 2).\tag{ii}
 \end{align*}
 \end{lem}
 
 \begin{proof}
(i) It suffices to give the details for
 \begin{equation}\label{eq2.3}
V(x) = \sum_{\frac X8 < n\le X} c_ne([n^c]x).
 \end{equation}
The left-hand side in (i) is
 \begin{gather*}
\int_B^{2B} \Bigg\{\sum_{\frac X8 < n\le X} |c_n|^2 + 2\sum_{\substack{
\frac X8 < n, n + j \le X\\
j \ne 0}} c_n \bar c_{n+j} e(([n^c] - [(n+j)^c])x)dx\Bigg\}\\[2mm]
\ll XB + \sum_{\frac X8 < n\le X} \ \sum_{j \le X} \ \frac 1{jX^{c-1}} \ll XB + X^{2-c}\mc L,
 \end{gather*}
since (assuming $X$ is large and $j > 0$) $[(n+j)^c] - [n^c] = (n+j)^c - n^c + O(1) \asymp jn^{c-1}$.
 \medskip

(ii) Again, we give details only for \eqref{eq2.3}. The left-hand side in (ii) is

 \begin{align*}
&\sum_{\frac X8 < n_j \le X \, (1\le j \le 4)} c_{n_1} c_{n_2} \bar c_{n_3} \bar c_{n_4} \int_B^{2B} e(([n_1^c] + [n_2^c] - [n_3^c] - [n_4^c])x)dx\\[2mm]
&\hskip .5in \le \sum_{\frac X8 < n_j\le X\, (j=1,\ldots, 4)} \min \left(B, \frac 1{|n_1^c + n_2^c - n_3^c - n_4^c + \t|}\right)\\[2mm]
&\hskip .5in = \sideset{}{_1}\sum \ , \ \text{ say.}
 \end{align*}
Here $\t$ depends on the $n_i$, $\t \in (-2,2)$. The number of terms in the sum with
 \[|n_1^c + n_2^c - n_3^c - n_4^c| \le 4^j\]
is
 \[\ll X^{\eta/4}(X^{4-c} 4^j + X^2)\]
for $4^j \ll X^c$, by \cite[Theorem 2]{robsar}. Thus
 \[\sideset{}{_1}\sum = \sum_{4^j \ll X^c} W_j\]
with $W_1$ corresponding to
 \[|n_1^c + n_2^c - n_3^c - n_4^c| \le 4\]
and $W_j$ corresponding to
 \[4^{j-1} < |n_1^c + n_2^c - n_3^c - n_4^c| \le 4^j.\]
We see that
 \[W_1 \ll X^{\eta/2}(X^2 + X^{4-c}) B 
 \ll X^{2+\eta/2}B\]
while for $j \ge 2$,
 \[W_j \ll X^{\eta/2} \min(4^{-j}, B)
 (X^2 + X^{4-c} 4^j).\]
The desired bound follows at once.
 \end{proof}
 
The following result abstracts the idea of Cai mentioned in Section \ref{sec:intro}.

 \begin{lem}\label{lem13}
Let $\mu$ be a complex Borel measure on $[X^{1-c}, K]$. Let $\l_1,\ldots,\l_N \in \mb R$. Let $a_n\left(\frac X8 < n \le X\right)$ be real numbers,
 
 \[\mc S(x) = \sum_{\frac X8 < n\le X} a_ne(\l_nx), \ \mc J(x) =
 \sum_{\frac X8 < n\le X} e(\l_nx).\]
Then
 \bigskip

\noindent (i) $\ds\left|\int_{X^{1-c}}^K \mc S(x) d\mu(x)\right|^2 \ll \bigg(\sum_{\frac X8 < n\le X} |a_n|^2\bigg) \int_{X^{1-c}}^K d\bar\mu(y) \int_{X^{1-c}}^K \mc J(x-y)d\mu(x)$.
 \bigskip

\noindent (ii) Suppose further that $a_n \ll \mc L$ and for some $U > 0$,
 \begin{equation}\label{eq2.4}
\mc J(x) \ll U + \mc L X^{1-c}|x|^{-1}\quad (0 < |x| \le 2K).
 \end{equation}
Then for any Borel measurable bounded function $G$ on $[X^{1-c}, K]$ we have
 \begin{align*}
\left|\int_{X^{1-c}}^K \mc S(x) G(x) dx\right|^2 &\ll \mc L^4 X^{2-c}\int_{X^{1-c}}^K |G(x)|^2dx\\[2mm]
&+ UX \mc L^2 \left(\int_{X^{1-c}}^K |G(x)|dx\right)^2. 
 \end{align*}
 \end{lem}
 
 \begin{proof}
(i) We have (summations over $n$ corresponding to $\frac X8 < n \le X$)
 \begin{align*}
\left|\int_{X^{1-c}}^K \mc S(x) d\mu(x)\right| &= \sum_n a_n \int_{X^{1-c}}^K e(\l_nx)d\mu(x)\\[2mm]
&\le \sum_n |a_n| \left|\int_{X^{1-c}}^K e(\l_nx) d\mu(x)\right|.
 \end{align*}
By Cauchy's inequality,
 \begin{align*}
\left|\int_{X^{1-c}}^K \mc S(x)d\mu(x)\right|^2 &\le \sum_n |a_n|^2 \sum_n \left|\ \int_{X^{1-c}}^K e(\l_nx) d\mu(x)\right|^2\\[2mm]
&= \sum_n |a_n|^2 \sum_n \int_{X^{1-c}}^K e(\l_nx)d\mu(x) \int_{X^{1-c}}^K e(-\l_ny)d\bar\mu(y)\\[2mm]
&= \sum_n |a_n|^2 \sum_n \ \int_{X^{1-c}}^K d\bar\mu(y) \int_{X^{1-c}}^K \mc J(x-y)d\mu(x). 
 \end{align*}
(ii) We apply (i) with $d\mu(x) = G(x)dx$. The right-hand side is
\newpage
 \begin{align*}
&\ll X\mc L^2 \int_{X^{1-c}}^K |G(y)| dy \int_{X^{1-c}}^K |G(x)| U\, dx\\[2mm]
&+ X\mc L^3 \int_{X^{1-c}}^K |G(y)| \int_{X^{1-c}}^K |G(x)| \min \left(X, \frac{X^{1-c}}{|x-y|}\right) dx\, dy.
 \end{align*}
It now suffices to show that
 \begin{align}
\int_{X^{1-c}}^K |G(y)| \int_{X^{1-c}}^K |G(x)| &\min \left(X, \frac{X^{1-c}}{|x-y|}\right) dx\, dy\label{eq2.5}\\[2mm]
&\ll X^{1-c} \mc L \int_{X^{1-c}}^K |G(x)|^2dx\notag.
 \end{align}
The left-hand side of \eqref{eq2.5} is
 \begin{align*}
&\le \frac 12 \int_{X^{1-c}}^K \int_{X^{1-c}}^K (|G(x)|^2 + |G(y)|^2) \min \left(X, \frac{X^{1-c}}{|x-y|}\right)dx\, dy\\[2mm]
&\quad = \int_{X^{1-c}}^K |G(x)|^2 \int_{X^{1-c}}^K \min \left(X, \frac{X^{1-c}}{|x-y|}\right)dy\, dx.
 \end{align*}
The contribution to the inner integral from $|y-x| \le X^{-c}$ is $\le 2X^{1-c}$ and the contribution from $|y-x| > X^{-c}$ is $\le 2X^{1-c}\log(KX^{c-1})$. Now \eqref{eq2.5} follows.
 \end{proof}

We write $d(n)$ for the divisor function.

 \begin{lem}\label{lem14}
Let $G$ be a complex function on $[X,2X]$. Let $u \ge 1$, $v$, $z$ be numbers satisfying $u^2 \le z$, $128uz^2 \le X$ and $2^{20}X \le v^3$. Then
 \[\sum_{\frac X8 < n\le X} \Lambda(n) G(n)\]
is a linear combination (with bounded coefficients) of $O(\mc L)$ sums of the form
 \[\sum_m a_m \sum_{\substack{n \ge z\\
 \frac X8 < mn\le X}} \, (\log n)^h G(mn)\]
with $h=0$ or 1, $|a_m| \le d^5(m)$ together with $O(\mc L^3)$ sums of the form
 \newpage
 \[\underset{\frac X8 < mn\le X}{\sum_m a_m \sum_{u \le n \le v}}
  b_n \, G(mn)\]
in which $|a_m| \le d(m)^5$, $|b_n| \le d(n)^5$.
 \end{lem}
 
 \begin{proof}
\cite[pp.~1367--1368]{hb}.
 \end{proof}

The following lemma encapsulates the `Harman sieve' in the version we need.

 \begin{lem}\label{lem15}
Let $w(\ldots)$ be a complex function with support on $\left[\frac X8, X\right] \cap \mb Z$, $|w(n)| \le X^{1/\eta}$ for all $n$. For $m \in \mb N$, $z \ge 2$ let
 \[S(m,z) = \sum_{(n,\, P(z))=1} w(mn).\]
Let $\a > 0$, $0 < \b \le 1/2$, $M \ge 1$, $Y > 0$. Suppose that whenever $|a_m| \le 1$, $|b_n| \le d(n)$, we have
 \begin{align}
&\sum_{m \le M} a_m \sum_n w(mn) \ll Y,\label{eq2.6}\\[2mm]
&\sum_{X^\a \le m \le X^{\a+\b}} a_m \sum_n b_n w(mn) \ll Y.\label{eq2.7}
 \end{align}
Let $|u_\ell| \le 1$, $|v_s| \le 1$, for $\ell \le R$, $s \le S$, also $u_\ell = 0$ for $(\ell, P(X^\eta)) > 1$, $v_s= 0$ for $(s, P(X^\eta))> 1$. Suppose that
 \[R < X^\a \ , \ S < MX^{-\a}.\]
Then
 \[\sum_{\ell \le R} \ \sum_{s \le S} u_\ell v_s
 S(\ell s, X^\b) \ll Y\mc L^3.\]
 \end{lem}

 \begin{proof}
\cite[Lemma 14]{bakwein2}.
 \end{proof}

 \begin{lem}\label{lem16}
Let $\a > 0$, $0 < \b \le 1/2$, $Y > 1$, $y > 0$. Suppose that whenever $|a_m| \le 1$, $|b_n| \le 1$ we have 
 \[S := \sum_{X^\a \le m \le X^{\a+\b}} \
 \sum_{\frac X8 < mn\le X} b_n e(y(mn)^c) \ll Y.\]
Let $|a_m| \le 1$, $|b_n| \le 1$. Let
 \[S_1 = \sideset{}{^*}\sum_{p_1,\ldots, p_s}\ \ 
 \underset{\substack{X^\a \le mp_1\ldots p_r \le X^{\a+\b}\\[1mm]
 \frac X8 < mn\, p_{r+1}\ldots p_s \le X}}{\sum a_m \sum b_n} e(y(mn\, p_1\ldots p_s)^c)\]
where the asterisk indicates that $X^\eta \le p_1 < p_2 < \cdots < p_r$ together with no more than $\eta^{-1}$ conditions of the form
 \[A(\mc F) \le \prod_{j\in \mc F} p_j \le B(\mc F)\]
$(\mc F \subseteq \{1, \ldots, s\}.)$ Then
 \[S_1 \ll YX^\eta.\]
Corresponding bounds hold when $S$, $S_1$ are replaced by sums containing (e.g.) $[(mn)^c]$ in place of $(mn)^c$.
 \end{lem}
 
 \begin{proof}
This is a variant of Lemma 10 of \cite{bakharriv}. Each condition implied by $\ast$ can be removed using repeatedly the truncated Perron formula
 \[\frac 1\pi \int_{-T}^T e^{i\a t}\ \frac{\sin t\b}t\, dt  
 = \begin{cases}
 1 + O(T^{-1}(\b-|\a|)^{-1} & \text{if $|\a| \le \b$}\\[2mm]
 O(T^{-1}(|\a| - \b)^{-1}) & \text{if $\a > \b$}.
 \end{cases}\]
We can keep the error term negligible by suitable choice of $T$, the main term being a multiple integral of a multiple sum with coefficients of absolute value at most $X^{\eta/2}$, and with no interaction between the summation variables. For more details see \cite[pp. 270--272]{bakharriv}.
 \end{proof}
 \bigskip
 
 \section{The minor arcs: small $x$ and large $x$.}\label{sec:minorarcs}
 
We can disregard the contribution to the minor arc from $x > K$ by \eqref{eq1.2}. In the present section we show that
 \begin{align}
\int_\tau^{X^{1-c}} |U(x)|^s |\Phi(x)|dx &\ll X^{s-c-3\eta}\ (3 \le s \le 4)\label{eq3.1}\\
\intertext{and}
\int_\tau^{X^{2-c}} |U(x)|^5 |\Phi(x)|ds &\ll X^{5-c-3\eta} \quad (s=5),\label{eq3.2}
 \end{align}
where $U(x)$ is any of $A(x)$, $B(x)$, $S(x)$, $S_1(x)$, $S^+(x)$, $T(x)$, $T^+(x)$. For Theorems \ref{thm1}--\ref{thm6}, this takes care of the (positive) left-hand part of the minor arc. (We need no separate discussion for the part of the minor arc in $(-\infty, -\tau)$, here or later.) This is not quite obvious for $s=2$, but see the discussion at the beginning of Section \ref{sec:minorarcthms12}.

 \begin{lem}\label{lem17}
Let $c < 2.1$ and $x \in (\tau, X^{2-c}]$. Let
 \begin{align}
V(x) &= \underset{\frac X8 < mn\le X}{\sum_{m \sim M} \ \sum_{n\sim N}}\ b_m c_n e(x(mn)^c)\label{eq3.3}
\intertext{or}
V(x) &= \underset{\frac X8 < mn\le X}{\sum_{m \sim M} \ \sum_{n\sim N}}\ b_m c_n e(x[(mn)^c]),\label{eq3.4}
 \end{align}
where $|b_m| \le 1$, $|c_n| \le 1$. Then
 \begin{equation}\label{eq3.5}
V(x) \ll X^{1-3\eta} \ \ \text{whenever } X^{10\eta}
\ll N \ll X^{1/2}. 
 \end{equation}
The bound \eqref{eq3.5} also holds when $b_n=1$ for all $n$ and $n \gg X^{1-10\eta}$.
 \end{lem}

 \begin{proof}
We prove this for \eqref{eq3.4}; the details for \eqref{eq3.3} are similar but simpler. We apply Lemma \ref{lem1} with
 \[a_n = \frac 1{X^{\eta/2}} \ \sum_{\ell s = n} \ \
 b_\ell c_s.\]
We take $H = X^{\frac 1{100}}$. Now it suffices to obtain
 \begin{equation}\label{eq3.6}
S(\g) := \underset{\frac X8 < mn\le X}{\sum_{m \sim M} \ \sum_{n\sim N}}\ b_m c_n e((h + \g)(mn)^c) \ll X^{1-7\eta/2}
 \end{equation}
with $\g =\{x\}$ or $-\{x\}$ and $|h| < H$, together with
 \begin{align}
\sum_{1 \le h \le H} \ \frac 1h\ \Bigg|\sum_{\frac X8 < n\le X} e(hn^c)\Bigg| &\ll X^{1-3\eta}\label{eq3.7}\\[2mm]
\intertext{and}
\sum_{h > H} \ \frac H{h^2}\ \Bigg|\sum_{\frac X8 < n\le X} e(hn^c)\Bigg| &\ll X^{1-3\eta}.\label{eq3.8} 
 \end{align}
In each case we use Lemma \ref{lem4} with $\ell=2$. Treating the simpler bounds \eqref{eq3.7}, \eqref{eq3.8} first, we bound the sum over $n$ in \eqref{eq3.7}, \eqref{eq3.8} by
 \begin{align*}
&\ll F^{1/14} X^{10/14} + F^{-1}X\\[2mm]
&\ll h^{1/14} X^{13/14},  
 \end{align*}
which yields \eqref{eq3.7}, \eqref{eq3.8}.

For \eqref{eq3.6}, take $Q = \eta N$. Arguing as in \cite[proof of Theorem 5]{rcb1}, we have
 \begin{equation}\label{eq3.9}
|S(\g)|^2 \ll \frac{X^2}Q + \frac XQ\ \sum_{q \le Q} \ \sum_{n \ll N} \Bigg|\sum_{m\in I(M)} e((h+\g)m^c((n+q)^c-n^c)))\Bigg|.
 \end{equation}

For the inner sum on the right-hand side we use Lemma \ref{lem4} with $\ell =2$, obtaining the bound
 \[\ll (qN^{-1} HX^c)^{\frac 1{14}} M^{\frac{10}{14}}
 + \frac M{|\g|q N^{-1}X^c}\,.\]

Thus
 \[|S(\g)|^2 \ll \frac{X^2}N + XN(HX^c)^{\frac 1{14}}
 + X^{2-7\eta}\]
since $|\g|X^c > X^{8\eta}$. This gives the desired bound \eqref{eq3.5}.

For the case $b_n=1$ identically, it suffices to add the bound
 \begin{equation}\label{eq3.10}
\sum_{n\in I(N)} e((h+\g)m^cn^c) \ll NX^{-4\eta} 
 \end{equation}
whenever $N \ge X^{1-10\eta}$. We bound the left-hand side by
 \[\ll (HX^c)^{\frac 1{14}}N^{\frac{10}{14}}
 + (xX^c)^{-1}N\]
and obtain \eqref{eq3.10} at once.

We now deduce \eqref{eq3.1} for $2 \le s\le 4$. We find that Lemma \ref{lem17} yields
 \begin{equation}\label{eq3.11}
U(x) \ll X^{1-2\eta} \quad (\tau < x < X^{1-c}) 
 \end{equation}
(here we require the reader to look ahead to the form of $S^+(x)$ and $T^+(x)$ in later sections, or else use Heath-Brown's identity if appropriate). Now (recalling Lemma \ref{lem12}(i)) a simple splitting-up argument yields (for some $B < X^\eta$)

 \begin{align*}
\int_\tau^{X^{1-c}} |U(x)|^s \Phi(x) dx &\ll \mc L X^{-c\eta} \sup_{y\in [\tau, K]} |U(y)|^{s-2} \int_\tau^K |U(x)|^2 dx\\[2mm]
&\ll \mc L X^{-c\eta} X^{s-2-3\eta} X^{2-c+\eta} \ll X^{s-c-3\eta}.
 \end{align*}
The argument for \eqref{eq3.2} is very similar using Lemma \ref{lem12} (ii).
 \end{proof}
 \bigskip

\section{Minor arc in Theorems \ref{thm1} and \ref{thm3}}\label{sec:minorarcthms12}

We shall show that
 \begin{equation}\label{eq4.1}
\int_{X^{1-c}}^K |S_1(x)|^4 dx \ll X^{4-c-3\eta}.
 \end{equation}
Since $\Phi(x) \ll X^{-c\eta}$, this reduces the integral in \eqref{eq1.4}, in effect, to the major arc in Theorem \ref{thm3}. For Theorem \ref{thm1}, let
 \[E_0(x) = \begin{cases}
 S^2(x) \Phi(x) & x \in [\tau, K]\\[1mm]
 0 & \text{otherwise.}
 \end{cases}\]
By Parseval's formula,
 \[\int_V^{2V} |\widehat E_0(R)|^2 dR < \int_\tau^K
 |E_0(x)|^2 dx \ll (V^{1/c})^{4-c-5c\eta}.\]
Hence
 \[\widehat E_0(R) = \int_{\mb R} S^2(x)
 \Phi(x) e(-Rx)dx\]
satisfies
 \[|\widehat E_0(R)| < V^{\frac 2c-1-2\eta}\]
except for a set of measure $O(V^{1-\eta})$ in $[V,2V]$. Again, this gives the desired `reduction to the major arc.'

To prove \eqref{eq4.1} we first apply Lemma \ref{lem13}(ii) with
 \[\mc S(x) = S_1(x), \mc J(x) = A(x), G(x)
 = \bar S_1(x) S_1(x)^2.\]
From Lemma \ref{lem12} and the Cauchy-Schwarz inequality,
 \begin{align*}
\int_\tau^K |G(y)|dy &\ll \mc L\left(\int_B^{2B} |S_1(x)|^2dx\right)^{1/2} \left(\int_B^{2B} |S_1(x)|^4 dx\right)^{1/2}\\
\intertext{(for some $B\in [\tau, K]$)}
 &\ll X^{\frac{5-c}2 + 2\eta}.
 \end{align*}
We shall show below that \eqref{eq2.4} holds with
 \begin{equation}\label{eq4.2}
U = X^{2-c-12\eta}
 \end{equation}
and that
 \begin{equation}\label{eq4.3}
S_1(x) \ll X^{(7-c)/6 - 5\eta}.
 \end{equation}
Hence

 \begin{align*}
\int_{X^{1-c}}^K |G(y)|^2dy &\ll X^{(7-c)/2-15\eta} \int_{X^{1-c}}^K |G(y)|dy\\[2mm]
&\ll X^{6-c-13\eta}.
 \end{align*}
Now Lemma \ref{lem13}(ii) yields
 \begin{align*}
\left(\int_\tau^K |S_1(x)|^4dx\right)^2 &\ll X^{2-c+6-c-10\eta}\\[2mm]
&\quad + X^{2-c-12\eta + 6-c + 2\eta} \ll X^{8-2c-10\eta} 
 \end{align*}
as required for \eqref{eq4.1}.

Turning to the proof of \eqref{eq4.2}, we apply the $B$ process first, followed by a partial summation and then Lemma \ref{lem8} with $k=5$. This is legitimate since the $B$ process produces a sum of the form
 \[\sum_{n\in I} e(yn^{c/(c-1)})\]
where
 \[yn^{c-1} \asymp F := xX^c;\]
and the five differentiations required in Lemma \ref{lem8} are permissible unless
 \begin{equation}\label{eq4.4}
\frac c{c-1} = m\in \mb N \ , \ m \le 4.
 \end{equation}
We have excluded $c = 4/3$, so \eqref{eq4.4} cannot hold. The error term in Lemma \ref{lem4} is
 \[\ll \mc L + F^{-1/2}X \ll X^{1/2}\]
(since $x \ge X^{1-c}$), which is acceptable. We may take
 \[U \ll XF^{-\frac 12} \ N_1^{1+\eta} \left\{
 (FN_1^{-5})^{\frac 1{20}} + N_1^{-\frac 1{20}} +
 (F_1N^{-5})^{-\frac 1{50}} N_1^{-\frac 1{10}}\right\}\]
where $N_1 = FX^{-1}$. Here
 \begin{align*}
XF^{-\frac 12} N_1^{1+\eta} (FN_1^{-5})^{\frac 1{20}} &\ll X^{\frac 14 + \eta} F^{3/10}\\[2mm]
&\ll X^{2-c-12\eta}
 \end{align*}
since $c < \frac{39}{29} < \frac{35}{26}$. Next,
 \[XF^{-1/2} N_1^{\frac{19}{20} + \eta} \ll
 X^{\frac{9c+1}{20} + 2\eta} \ll X^{2-c-12\eta}\]
since $c < \frac{39}{29}$. Finally
 \[XF^{-\frac 12}(FN_1^{-5})^{-\frac 1{50}} 
 N_1^{\frac 9{10}+\eta} \ll X^\eta F^{24/50}
 \ll X^{2-c-12\eta}\]
since $c < \frac{39}{29} < \frac{50}{37}$.
 \medskip

We now use Lemma \ref{lem14} to prove \eqref{eq4.3}. Here and below, we take $G(n) = e(xn^c)$ in Lemma \ref{lem14}. A \textit{Type I} sum will be of the form
 \[S_I(x) = \underset{\frac X8 < mn\le X}{\sum_{m\sim M} a_m 
 \sum_{n\sim N}} e((mn)^cx)\]
and a \textit{Type II} sum will be of the form
 \[S_{II}(x) = \underset{\frac X8 < mn\le X}{\sum_{m\sim M} a_m 
 \sum_{n\sim N}} b_n\ e((mn)^cx)\]
Here $|a_m| \le 1$, $|b_n| \le 1$. Taking
 \[v = X^{0.36} \ , \ u = X^{0.12} \ , \
 z = X^{0.35}\]
in Lemma \ref{lem14}, it suffices to show that
 \begin{align}
S_I(x) &\ll X^{(7-c)/6-6\eta} \quad \text{for } N \ge z\label{eq4.5}\\
\intertext{and}
S_{II}(x) &\ll X^{(7-c)/6-6\eta} \quad \text{for } u \le N \le v.\label{eq4.6}
 \end{align}
 \newpage

For \eqref{eq4.5}, we appeal to Lemma \ref{lem10}. We have $(7-c)/6 > 0.942$. The first two terms in the bound for $S_I(x)$ are acceptable, while
 \[XF^{-1} \ll 1.\]
Next, for $N \ge z$,
 \[F^{1/8} X^{13/16} N^{-1/8} \ll X^{1.345/8+13/16-0.35/8}
 \ll X^{0.94}.\]
Finally, we have a term that is
 \begin{align*}
&\ll (FX^4)^{1/6} + (FX^5N^{-2})^{1/6}\\[2mm]
&\ll X^{(1.345+4)/6} + X^{(1.345+5-0.7)/6}\ll X^{0.941}.
 \end{align*}

For \eqref{eq4.6}, we apply the obvious variant of \eqref{eq3.9}, taking $Q = X^{0.116}$ to give an acceptable term $X^2/Q$. It remains to show that for $q \le Q$, $n \le N$,
 \begin{equation}\label{eq4.7}
S_{n,q} : = \sum_{m\sim M} e(xm^c((n+q)^c-n^c)) \ll MQ^{-1}. 
 \end{equation}
Here $F$ is replaced by $F_1 := xqX^cN^{-1}$. We apply Lemma \ref{lem5} (ii) with $\ell=4$ to obtain
 \[S_{n,q} \ll F_1^{13/31}M^{3/31}\ll MX^{-0.116}\]
since $X^{28/31} N^{-15/31} > X^{0.729} > X^{(39/29 + .116)13/31 + .116}$. However, the six differentiations are only permissible when
 \[c\ne 1 + \frac 1m\]
where $m \le 4$. We excluded $m=3$, so we now need to treat $c = \frac 54$ separately. Here we use Lemma \ref{lem5} (ii) with $\ell=3$; the five differentiations are permissible and
 \[S_{n,q} \ll F_1^{11/30}M^{1/6} \ll MX^{0.116}\]
by a similar calculation. This completes the discussion of the minor arc.
 \bigskip
 
 \section{Minor arc in Theorem \ref{thm2}.}\label{sec:minorarcthm3}
 
 We shall set up a suitable function $\rho^+$ based on Type I and Type II information. To obtain a negligible contribution of the minor arc we require
 \begin{equation}\label{eq5.1}
\int_{X^{1-c}}^K S(x) G(x) dx \ll X^{3-c-3\eta} 
 \end{equation}
for the two functions
 \begin{equation}\label{eq5.2}
G(x) = S^+(x) S(x) \Phi(x) e(-Rx), \ S^+(x)^2 \Phi(x) e(-Rx). 
 \end{equation}
It will suffice to show that
 \begin{align}
A(x) &\ll X^{\frac c2 + \eta} + X^{1-c} |x|^{-1} \quad (|x| < 2X^{2\eta})\label{eq5.3}\\
\intertext{and}
S^+(x) &\ll X^{\frac{3 - c}2 - 4\eta} \quad (X^{1-c} \le x \le K).\label{eq5.4}
 \end{align}
We then apply Lemma \ref{lem13} (ii) with $\mc S(x) = S(x)$, $\mc T(x) = A(x)$, $G(x)$ as in \eqref{eq5.2} so that
 \begin{align*}
\int_{X^{1-c}}^K |G(x)|dx &\ll X^{1+\eta}\\[2mm]
\intertext{and, using \eqref{eq5.4},}
\int_{X^{1-c}}^K |G(x)|^2dx &\ll X^{3 - c - 8\eta}\int_{X^{1-c}}^K |G(x)|dx\\[2mm]
&\ll X^{4-c-7\eta}.
 \end{align*}
Thus
 \begin{align*}
\left(\int_{X^{1-c}}^K S(x) G(x) dx\right)^2 &\ll \mc L^4X^{2-c} X^{4-c-7\eta} + \mc L^2 X^{\frac c2 + 1 + \eta + 2(1+2\eta)}\\[2mm]
&\ll X^{6 - 2c - 6\eta} 
 \end{align*}
using $c < 6/5$, which proves \eqref{eq5.1}.

To obtain \eqref{eq5.3} we use the Kusmin-Landau theorem if $X^{c-1}|x| < \eta$. Otherwise, we use the $B$ process, giving a main term
 \[\ll F^{\frac 12} \ll X^{\frac c2 + \eta}\]
where $F = xX^c$, and error terms
 \[\ll F^{-1/2}X + \mc L \ll X^{1/2}.\]

Aiming towards the definition of $S^+(x)$, we claim that Type II sums are $\ll X^{0.9}$ for either of the alternatives
 \begin{align}
X^{1/5} &\ll N \ll X^{29/105}\label{eq5.5}\\
\intertext{and}
X^{1/3} &\ll N \ll X^{11/25}.\label{eq5.6} 
 \end{align}
For \eqref{eq5.5} we begin with \eqref{eq3.9}, replacing $h + \b$ by $x$, and taking $Q = X^{0.2} \ll N$. It remains to show that for given $Q_1\in \left[\frac 12, Q\right]$ and $n\sim N$, we have
 \[S^* := \sum_{q\sim Q_1} \ \sum_{n\in I(M)}
 e(x((n+q)^c-n^c)m^c) \ll Q_1 MX^{-1/5}.\]
Following the analysis on pp. 171--172 of \cite{bakwein2}, we find that for some $q \sim Q_1$, and $R$ at our disposal with $R \ll N_1$, and some $r\sim R$, we have
 \begin{equation}\label{eq5.7}
\frac{S^{*4}}{\mc L^4} \ll \frac{N^5M^4}{FQ} + N^4M^2 + \frac{X^4NQ_1}{FQ^2} \left(\frac{N_1^2}R + N_1|S(n,q,r)|
\right). 
 \end{equation}
Here $N_1 \asymp FQ_1/X \ll X^{0.4}$,
 \begin{align*}
t(n,q) &= (n+q)^c - (n-q)^c,\\[2mm]
t_1(n_1,r) &= (n_1+r)^{c/(c-1)} - (n_1-r)^{c/(c-1)}, 
 \end{align*}
and we define
 \[S(n,q,r) = \sum_{n_1\in I(r)} e\left(
 C(xX^ct(n,q))^{\frac 1{1-c}}\ t_1(n_1,r)\right)\]
with $I(r)$ a subinterval of $[N_1,2N_1]$. We choose $R$ so that
 \[\frac{M^4 N^5 Q_1 N_1^2}{FQ^2R} = \frac{X^4}{Q^2},\]
that is,
 \[R = \frac{NQ_1N_1^2}F \asymp \frac{FNQ_1^3}{X^2}.\]
We have $R \ll N_1$ since $NQ_1^2 \ll NQ^2 \ll X$.

The terms $N^5M^4/FQ$ and $N^4M^2$ in \eqref{eq5.8} are $\ll X^4/Q^2$ since
 \begin{align*}
\frac{N^5M^4}{FQ} \ \frac{Q^2}{X^4} &= \frac{NQ}F \ll \frac{NQ}X \ll 1,\\[2mm]
\frac{N^4M^2Q^2}{X^4} &\ll \frac{Q^2}{M^2} \ll 1. 
 \end{align*}
For $S(n,q,r)$, it suffices to show that
 \[S(n,q,r) \ll X^{0.6}/N.\]
For then
 \begin{align*}
\frac{X^4}{FQ^2}\, NQ_1N_1 \, S(n,q,r) &\ll \frac{X^4}{FQ^2}\, X^{0.6}Q_1\, \frac{FQ_1}X\\[2mm]
 &\ll \frac{X^4}{Q^2}\, .
 \end{align*}

We apply Lemma \ref{lem7} to $S(n,q,r)$ with (taking $2\eta < 1.2-c$). We have
 \[T \asymp \frac{FQ_1}N \ \frac r{N_1} \asymp
 \frac{Xr}N \ ; \ T \ll \frac{XR}N \asymp \frac{FQ_1^3}X
 \ll X^{0.8-2\eta}.\]
Provided that
 \[T^{17/42} \le N_1 \le T^{25/42},\]
we obtain
 \begin{align*}
S(n,q,r) &\ll X^\eta T^{13/84} N_1^{1/2}\\[2mm]
&\ll X^{\eta + \frac{13}{84}\, \frac 45 + \frac 15 - \eta} \ll \frac{X^{0.6}}N \ \ (N \ll X^{29/105}). 
 \end{align*}

We certainly have
 \[N_1 \le T^{25/42}\]
since
 \[X^{0.4} < \left(\frac XN\right)^{25/42} \kern -10pt X^{-\eta}
 \ \text{ as }\ N< X^{0.3} < X^{1-0.4 \times \frac{42}{25}
 - 2\eta}.\]
We may have $N_1 < T^{17/42}$. In this case we apply Lemma \ref{lem4} with $\ell =2$ to $S(n,q,r)$. The term $T^{-1}N_1$ is $\ll 1$, so that
 \[S(n,q,r) \ll T^{1/14} N_1^{10/14} \ll
 T^{\frac 1{14} + \frac{170}{588}} < \frac{X^{0.6}}N\]
since $T < X^{4/5}$, $N < X^{3/10}$. This completes the proof that $S_{II} \ll X^{0.9}$ when \eqref{eq5.5} holds.

Now suppose that \eqref{eq5.6} holds. We apply Lemma \ref{lem11}. Five of the terms $U_1, U_2, \ldots, U_9$ on the right-hand side of \eqref{eq2.2} are acceptable for $N \ll X^{0.5}$:
 \begin{align*}
U_1 &\ll F^{1/20} N^{9/40} X^{29/40} \ , \ \frac{1.2}{20} + \frac{9/2}{40} + \frac{29}{40} = 0.8\ldots,\\[2mm]
U_2 &\ll F^{3/46} N^{11/46} X^{32/46} \ ,\ \frac{3.6}{46} + \frac{11/2}{46} + \frac{32}{46} = 0.8\ldots,\\[2mm]
U_3 &\ll F^{1/10} N^{3/10} X^{3/5} \ , \frac{1.2}{10} + \frac{3/2}{10} + \frac 35 = 0.87,\\[2mm]
U_5 &\ll F^{1/11} N^{1/33} X^{17/22} \ , \ \frac{1.2}{11} + \frac{1/2}{33} + \frac{17}{22} = 0.8\ldots,\\[2mm]
U_7 &\ll F^{1/5} N^{1/10} X^{3/5} \ , \ \frac{1.2}5 + \frac{1/2}{10} + \frac 35 = 0.89.
 \end{align*}
Also $U_8 \ll XN^{-1/2} \ll X^{0.9}$ for $N > X^{0.2}$. For the remaining terms,
 \begin{align*}
U_4 &\ll F^{3/28} N^{5/56} X^{41/56} \ll X^{0.9} \ \text{ for } \ N \ll X^{0.44},\\[2mm]
U_6 &=  F^{2/21} X^{17/21} N^{-1/14} \ll X^{0.9} \ \text{ for } \ N \gg X^{1/3},\\
\intertext{while the bound}
U_9 &= F^{1/8} X^{3/4} \ll X^{0.9} \quad (F \ll X^{6/5})
 \end{align*}
actually determines our range of $c$.

Finally we consider Type I sums, using
 \[S_I \ll M\, F^{1/14} N^{10/14},\]
which follows from Lemma \ref{lem4} with $\ell = 2$. Here
 \[MF^{1/14} N^{10/14} \ll X^{\frac{10}{14} + 
 \frac{1.2}{14}} M^{\frac 4{14}} \ll X^{0.9} \ 
 \text{ for } \ M \ll X^{0.35}.\]
If $X^{0.35} \ll M \ll X^{0.44}$ we treat $S_I$ as a Type II sum. Hence
 \[S_I \ll X^{0.9} \ \text{ for } \ M \ll X^{0.44}.\]

We now apply Lemma \ref{lem15}, taking $w(n) = e(xn^c)$, $\a = \frac 13$ and $\b = \frac{11}{25} - \frac 13 = \frac 8{75}$, $M = X^{0.44}$, $S=1$. Thus
 \[\sum_{\ell \le X^{11/25}} u_\ell S(\ell, X^\b)
 \ll X^{0.9}\mc L^3\]
for any coefficients $u_\ell$ with $|u_\ell| \le 1$, $u_\ell = 0$ for $(\ell, P(X^\eta)) > 1$. (For $X^{1/3} < \ell \le X$ this uses Lemma \ref{lem16}.)  We use Buchstab's identity
 \[\rho(u,z) = \rho(u,w) - \sum_{w\le p < z} \
 \rho\left(\frac up, p\right) \quad (2 \le w < z).\]
Multiplying by $e(xn^c)$ and summing over $n$, we obtain
 \begin{align}
S(x) &= \sum_{\frac X8 < n\le X} \rho(n, (3X)^{1/2})\, e(xn^c)\label{eq5.8}\\[2mm]
&= \sum_{\frac X8 < n\le X} \rho(n,X^\b)\, e(xn^c) - \sum_{X^\b \le p_1 < (3X)^{1/2}} \ \sum_{\substack{\frac X8 < p_1n \le X\\
(n,P(p_1))=1}} e(x(p_1n)^c).\notag
 \end{align}
In writing sums over primes, we define $\a_j$ by $p_j = X^{\a_j}$. We introduce intervals
 \begin{align*}
I_1 &= \left[\b, \frac 15\right) \ , \ I_2 = \left[\frac 15, \frac{29}{105}\right) \ , \ I_3 = \left[\frac{29}{105}\, , \, \frac 13\right) \ , \ I_4 = \left[\frac 13, \, \frac{11}{25}\right),\\[2mm]
I_5 &= \left[\frac{11}{25}, \, \frac 12 + \frac{\log 3}{2\log X}\right).
 \end{align*}
Let
 \[S_j(x) = \sum_{\a_1\in I_j} \ \sum_{\substack{
 \frac X8 < p_1n \le X\\
 (n,\, P(p_1))=1}} e(x(p_1n)^c) \quad (1 \le j \le 4)\]
and $\ds S_0(x) = \sum_{\substack{\frac X8 < n\le X\\
(n, \, P(X^\b))=1}} e(xn^c)$,
 \[D_1(x) = \sum_{\a_1\in I_5} \ \sum_{\substack{
 \frac X8 < p_1n \le X\\
 (n,\, P(p_1))=1}} e(x(p_1n)^c).\]
From \eqref{eq5.8}, we have

 \begin{equation}\label{eq5.9}
S(x) = S_0(x) - \sum_{j=1}^4 S_j(x) - D_1(x).
 \end{equation}

We use Buchstab's identity again for $S_1(x)$:
 \begin{align}
S_1(x) &= \sum_{\a_1\in I_1} \ \sum_{\substack{
\frac X8 < p_1n \le X\\
(n, \, P(X^\b))=1}} e(x(p_1n)^c)\label{eq5.10}\\[2mm]
&\qquad - \sum_{\a_1\in I_1} \ \sum_{\b \le \a_2 < \a_1} \ \sum_{\substack{
\frac X8 < p_1p_2n\le X\\
(n,\, P(p_2))=1}} e(x(p_1p_2n)^c)\notag\\[2mm]
&= S_5(x) - S_6(x), \text{ say}.\notag 
 \end{align}
Next,
 \[S_6(x) = S_7(x) - S_8(x)\]
where
 \begin{align}
S_7(x) &= \sum_{\a_1\in I_1} \ \sum_{\b \le \a_2 < \a_1} \
 \sum_{\substack{
 \frac X8 < p_1p_2 n\le X\\
 (n,\, P(X^\b))=1}} e(x(p_1p_2n)^c),\notag\\[2mm]
S_8(x) &= \sum_{\a_1\in I_1} \ \sum_{\b \le \a_3 < \a_2 < \a_1} \
 \sum_{\substack{
 \frac X8 < p_1p_2p_3 n\le X\\
 (n,\, P(p_3))=1}} e(x(p_1p_2p_3n)^c).\label{eq5.11}
 \end{align}
We write $D_2(x)$ for the part of the right-hand side of \eqref{eq5.11} with $\a_1 + \a_3 \in I_3$, $\a_2 + \a_3 \in I_3$, $\a_1 + \a_2 + \a_3 \in I_5$, and define $K_2(x)$ by
 \begin{equation}\label{eq5.12}
S_8(x) = D_2(x) + K_2(x).
 \end{equation}

Next, considering the possible decompositions $n=p_2$ and $n=p_2p_3$ in the definition of $S_3(x)$, we have
 \begin{equation}\label{eq5.13}
S_3(x) = D_3(x) + S_9(x)
 \end{equation}
where
 \[D_3(x) = \underset{1 \le \a_1 + \a_2 < 1 + 
 \frac{\log 2X}{\log X}}{\sum_{\a_1 \in I_3} \ 
 \sum_{\a_2 > \a_1}} e(x(p_1p_2)^c),\]
 
 \[S_9(x) = \underset{1 \le \a_1 + \a_2 + \a_3 < 1 
 + \frac{\log 2X}{\log X}}{\sum_{\a_1 \in I_3} \ 
 \sum_{\a_3 \ge \a_2 > \a_1}} e(x(p_1p_2p_3)^c).\]

Finally,
 \begin{equation}\label{eq5.14}
S_9(x) = D_4(x) + K_4(x),
 \end{equation}
where $D_4(x)$ is the part of the sum defining $S_9(x)$ for which $\a_1 + \a_2 \le \frac{14}{25} + \frac{\log 2}{\log X}$, and $K_4 : = S_9 - D_4$. Combining \eqref{eq5.9}--\eqref{eq5.14}, our decomposition of $S(x)$ is
 \[S = S_0 - S_5 + S_7 - D_2 - K_2 - S_2 - D_3 - D_4 -
 K_4 - S_4 - D_1.\]
We define
 \begin{align}
S^+ &= S_0 - S_5 + S_7 - K_2 - S_2 - K_4 - S_4\label{eq5.15}\\[2mm]
&= S + \sum_{j=1}^4 \, D_j.\notag 
 \end{align}
We observe firstly that
 \[S^+(x) = \sum_{\frac X8 < n\le X} \rho^+(n)\, e(xn^c)\]
with $\rho^+ \ge \rho$, since the $D_j$ have non-negative coefficients. Secondly, all of $S_0$, $S_5$, $S_7$, $K_2$, $S_2$, $K_4$, $S_4$ have values $\ll X^{3-c-6\eta}$, hence so does $S^+$. For $S_0$, $S_5$, $S_7$ this follows from Lemma \ref{lem15}. For $S_2$, $S_4$, $K_2$ and $K_4$ we appeal to Lemma \ref{lem16} and the following observations.
 \medskip
 
 \begin{enumerate}
\item[(i)] If $\b \le \a_3 < \a_2 < \a_1 < 1/5$ and $\a_2 + \a_3\not\in I_3$, then $\a_2 + \a_3 > \frac{16}{75} > \frac 15$, $\a_2 + \a_3 < 2/5$. Hence $\a_2 + \a_3 \in I_2 \cup I_4$. Similarly for $\a_1 + \a_2$.
 \bigskip

\item[(ii)] If $\b \le \a_3 < \a_2 < \a_1 < 1/5$ and $\a_2 + \a_3 \in I_3$, $\a_1 + \a_3 \in I_3$, $\a_1 + \a_2 + \a_3 \not\in I_5$, then $\a_1 + \a_2 + \a_3 \le \frac 32 \cdot \frac 13 = \frac 12$, hence $\a_1 + \a_2 + \a_3 < \frac{11}{25}$, while $\a_1 + \a_2 + \a_3 \ge \frac 32 \cdot \frac{29}{105} > \frac 13$; $\a_1 + \a_2 + \a_3 \in I_4$.
 \bigskip

\item[(iii)] If $\a_1 < \a_2 \le \a_3$, $\frac{29}{105} \le \a_1 < \frac 13$, $1 \le \a_1 + \a_2 + \a_3 < \frac{\log 2X}{\log X}$, and $\a_1 + \a_2 \ge \frac{14}{25} + \frac{\log 2}{\log X}$, then $\a_3 \le \frac{11}{25}$. Moreover $
\a_3 \ge \frac 13\, (\a_1 + \a_2 + \a_3) \ge \frac 13$, hence $\a_3 \in I_4$.
 \end{enumerate}

In Section \ref{sec:minorarcthms1_6} we shall quantify the contribution of the functions $D_j(x)$ to the integral in \eqref{eq1.6}; similarly for Theorem \ref{thm5,thm6}.
 \bigskip
 
 \section{Minor arc in Theorem \ref{thm5}.}\label{sec:minorarcthm5}

In the present section we show that for $c < \frac{3581}{3106}$, we have
 \begin{equation}\label{eq6.1}
\int_{X^{1-c}}^{\frac 12} T(x) G(x)dx \ll X^{3-c-3\eta},
 \end{equation}
where
 \[G(x) = T^2(x) \Phi(x) e(-Rx).\]
We apply Lemma \ref{lem13} (ii) with $\mc S(x) = T(x)$, $\mc J(x) = B(x)$. To prove \eqref{eq6.1} it suffices to show that
 \begin{align}
B(x) &\ll X^{3-2c-15\eta} + \mc L X^{1-c} |x|^{-1} \quad (0 < |x| < 2X^{2\eta})\label{eq6.2}\\
\intertext{and}
T(x) &\ll X^{\frac{3-c}2 - 5\eta} \quad \left(X^{1-c} \le x < \frac 12\right).\label{eq6.3}
 \end{align}
For then Lemma \ref{lem13} (ii) (with $G(x) = 0$ for $x > \frac 12$) yields
 \begin{align*}
\left|\int_{X^{1-c}}^{\frac 12} T(x) G(x)\, dx\right|^2 &\ll X^{2-c} \mc L^4 \max_{x\in \left[X^{1-c},\frac 12\right]} |T(x)|^2 \int_{X^{1-c}}^{\frac 12} |T(x)|^2 dx\\[2mm]
&\qquad + X^{4-2c-15\eta} \mc L^2\left(\int_{X^{1-c}}^{\frac 12} |T(x)|^2dx\right)^2. 
 \end{align*}
The first summand on the right-hand side is
 \[\ll X^{2-c+3-c-10\eta + 1 + 3\eta} \ll
 X^{6-2c-7\eta}\]
by \eqref{eq6.3} and Lemma \ref{lem12} (i). The second summand is
 \[\ll X^{4-2c-14\eta} X^{2(1+2\eta)}
 \ll X^{6-2c-6\eta}\]
by \eqref{eq6.2} and Lemma \ref{lem12} (i).

For \eqref{eq6.2}, we use Lemma \ref{lem1} with $a_n = 1$. We take
 \[H = X^{2c-2+16\eta}.\]
Since $\{x\} = x$ in the sum, our objective is to show that we have

 \begin{align}
\sum_{0 \le h \le H} \min \left(1, \frac 1h\right) \Bigg|\sum_{\frac X8 < n\le X} e((h\pm x)n^c)\Bigg| &\ll X^{3-2c-16\eta} + X^{1-c} |x|^{-1},\label{eq6.4}\\[2mm]
\sum_{1\le h \le H} \frac 1h \Bigg|\sum_{\frac X8 < n\le X} e(hn^c)\Bigg| &\ll X^{3-2c-16\eta}\label{eq6.5}\\
\intertext{and}
\sum_{h > H} \ \frac H{h^2} \Bigg|\sum_{\frac X8 < n\le X} e(hn^c)\Bigg| &\ll X^{3-2c-16\eta}\label{eq6.6}.
 \end{align}

We begin with the contribution from $h=0$ in \eqref{eq6.4}.
If $X^{c-1}|x| < \eta$, we obtain the desired bound from the Kusmin-Landau theorem. Otherwise Lemma \ref{lem5} (ii) with $\ell = 2$ yields the bound
 \[\ll (x\, X^c)^{2/7} X^{2/7},\]
which is acceptable since $16c < 19$.

For the terms in \eqref{eq6.6} with $h \ge X^{3/2-c}$ we use Lemma \ref{lem5} (ii) with $\ell = 2$. It is clear that these sums produce a contribution
 \[\ll X^a \ll X^{3-2c}\]
where
 \begin{equation}\label{eq6.7}
a = 2c - 2 + \eta - \frac 57\, \left(\frac 32 - c\right) + \frac{2c+2}7,
 \end{equation}
and $a < 3 -2c$ follows from $c < \frac{81}{70}$.

We can treat together the terms in \eqref{eq6.4} with $1 \le h \le H$, the sum \eqref{eq6.5}, and the remaining part of \eqref{eq6.6} by estimating (for $\frac 12 \le H_1 \ll H^{\frac 32 - c}$, $H_1 = 2^j$)
 \[S(H_1) : = H_1^{-1} \sum_{h\sim H_1} \Bigg|
 \sum_{\frac X8 < n\le X} e((h+\g)n^c)\Bigg|\]
where $\g \in \{x, -x, 0\}$. Let $F = H_1X^c$. We apply the $B$ process, followed by the $A$ process, to the sum over $n$, choosing
 \[Q = H_1X^{5c-6+34\eta}\]
so that
 \[XF^{-1/2} \left(\frac{FX^{-1}}{Q^{1/2}}\right)
 \ll X^{3-2c-17\eta}.\]
The errors from the $B$ process contribute (for some $H_1$)
 \[\ll \mc L H_1^{-1} \sum_{h\sim H_1} \left(
 X F^{-1/2} + 1\right) \ll X^{1-c/2} 
 H_1^{-\frac 12} + \mc L,\]
which is acceptable.

After the $A$ process we arrive at sums
 \[S^*(h) = \sum_{n\in I} e\left((h+\g)^{\frac 1{1-c}}
 \left((n+q)^{\frac c{c-1}} - n^{\frac c{c-1}}\right)
 \right),\]
where the interval $I$ has endpoints $\asymp FX^{-1} \asymp H_1X^{c-1} = N_1$, say. It suffices to show that, for fixed $q\in [1,Q]$,
 \[\sum_{h\sim H_1} |S^*(h)| \ll H_1 N_1 Q^{-1}.\]
We apply Lemma \ref{lem6}, with $y_h = \left(\frac{h+\g}{H_1 + \g}\right)^{1-c}$. This gives, with $F_1 = H_1 X^c qN_1^{-1} \asymp qX$, and assuming initially that $N_1 \in [F_1^{1/3}, F_1^{1/2}]$,
 \begin{equation}\label{eq6.8}
\sum_{h\sim H_1} |S^*(h)| \ll H_1^{\frac{319}{345}} N_1^{\frac{449}{690}} F_1^{\frac{63}{690}+\eta} + H_1 N_1^{\frac 12} F_1^{\frac{141}{950}+\eta}.
 \end{equation}
It suffices to show that
 \begin{align}
&N_1^{\frac{449}{690}} F_1^{\frac{63}{690}} \ll H_1^{\frac{26}{345}} X^{-2\eta} N_1 Q^{-1}\label{eq6.9}\\
\intertext{and}
&N_1^{\frac 12} F_1^{\frac{141}{950}} N_1^{-1} Q \ll X^{-2\eta}\, .\label{eq6.10}
 \end{align}
The worst case in each of \eqref{eq6.8}, \eqref{eq6.9} is $q = Q$, $H_1=H$. (The factor lost for $H_1 > H$ is outweighed by the factor $HH_1^{-1}$ arising from $Hh^{-2}$). For \eqref{eq6.8} we require
 \begin{equation}\label{eq6.11}
(HX^{c-1})^{\frac{449}{690}} (HX^{5c-5})^{\frac{63}{690}} H^{-\frac{26}{345}} (HX^{c-1})^{-1} HX^{5c-6} \ll X^{-2\eta}. 
 \end{equation}
It may be verified that \eqref{eq6.10} holds with something to spare for $c < \frac{3581}{3106}$.

For \eqref{eq6.9} we must show
 \[(HX^{c-1})^{-\frac 12} (QX)^{\frac{141}{950}}
 HX^{5c-6} \ll X^{-2\eta}.\]
This follows after a short computation from $c < \frac{3581}{3106}$, determining our upper bound for $c$.

We certainly have $N_1 \le F_1^{1/2}$ (using $c < 7/6$). If $N_1 < F^{1/3}$, we apply Lemma \ref{lem5} (i) with $\ell=2$:
 \[S^*(h) \ll F_1^{\frac 1{14}} N_1^{10/14}
 \ll F_1^{13/42}.\]
It suffices to show that $S^*(h) \ll N_1 Q^{-1}$, or
 \begin{equation}\label{eq6.12}
F_1^{13/42} N_1^{-1}Q \ll X^{-2\eta}.
 \end{equation}
The worst case is $q = Q$, $H_1 = H$ and in this case the left-hand side of \eqref{eq6.11} is
 \[\ll (X^{7c-7})^{\frac{13}{42}} X^{4c-5+C\eta}
 \ll X^{-2\eta}.\]
This completes the discussion of \eqref{eq6.2}.

In view of Lemma \ref{lem15}, in order to prove \eqref{eq6.3} it suffices to show that
 \begin{align}
S_{II} &: = \underset{\frac X8 < m\ell \le X}{\sum_{m\sim M} \ \sum_{\ell \sim N}}\ b_m c_\ell e(x[(m\ell)^c]) \ll X^{0.92353}\label{eq6.13}\\
\intertext{for $X^{0.16} \ll N \ll X^{0.38}$, and that}
S_I &:= \underset{\frac X8 < m\ell \le X}{\sum_{m\sim M} \ \sum_{\ell \sim N}}\ b_m e(x[(m\ell)^c]) \ll X^{0.92353}\label{eq6.14}
 \end{align}
for $N \gg X^{0.38}$. In both cases we apply Lemma \ref{lem1} (adapted to allow $a_n \ll X^\eta$) with (e.g.)
 \[a_n = \sum_{m\ell = n} \ b_mc_\ell \ \
 \text{ in case \eqref{eq6.12}}.\]
We choose $H = X^{0.07648}$ in both cases.

For \eqref{eq6.13}, it suffices to show that
 \begin{equation}\label{eq6.15}
\underset{\frac X8 < m\ell \le X}{\sum_{m\sim M} \ \sum_{\ell \sim N}}\ b_m c_\ell e(x\, m^c\, n^c) \ll X^{0.92353}
 \end{equation}
(corresponding to $h=0$ in Lemma \ref{lem1}) and that
 \begin{equation}\label{eq6.16}
\underset{\frac X8 < m\ell \le X}{\sum_{m\sim M} \ \sum_{\ell \sim N}}\ b_m c_\ell e ((h+\g) m^c\, n^c) \ll X^{0.92352}.
 \end{equation}
(The terms with $h > H$ in Lemma \ref{lem1} are covered by \eqref{eq6.7}.) Proceeding as in \eqref{eq3.9}, and taking $Q = X^{0.15296}$, we require the bound
 \[S(q,M) := \sum_{m\sim M} e(\l\, m^c((n+q)^c
 -m^c)) \ll MQ^{-1}\]
($1 \le q \le Q$, $Q \ll N \ll X^{0.38}$). Here $\l \in \{x, h+\g\}$. We apply Lemma \ref{lem5} (ii) with $\ell =2$ to obtain
 \[S(q, M) \ll (qN^{-1} \l X^c)^{2/7} M^{2/7} +
 (q N^{-1} \l\, X^c)^{-1}.\]
The first term is bounded by $MQ^{-1}$, as we easily verify. Since $q N^{-1} x\, X^c \ge XN^{-1}$, the second term is acceptable.

For \eqref{eq6.14} it suffices with the same $\l$ to show that
 \[\underset{\frac X8 < mn \le X}{\sum_{m\sim M} \ 
 \sum_{n \sim N}}\ a_m e(\l m^c\, n^c) \ll
 X^{0.92352}\]
whenever $N \ge X^{0.38}$. We appeal to Lemma \ref{lem10}, with $F = \l\, X^c$. As above, the term $X^{1+\eta}F^{-1}$ causes no difficulty, and the terms $X^{11/12 + \eta}$, $X^{1+\eta}N^{-\frac 12}$ are also acceptable. We have
 \begin{align*}
F^{1/8} X^{13/16} N^{-1/8} &\ll  \left(X^{\frac{3581}{3106} + 0.07648}\right)^{\frac 18} X^{\frac{13}{16} - \frac{0.38}8} \ll X^{0.92},\\
\intertext{and}
(FX^5 N^{-1}N_0^{-1})^{\frac 16} &\ll (FX^{5-0.76})^{1/6}\\[2mm]
&\ll X^{\left(\frac{3581}{3106} + 0.07648 + 4.24\right)/6} \ll X^{0.92}.
 \end{align*}
This completes the proof of \eqref{eq6.3} and the discussion of the minor arc.
 \bigskip
 
 \section{Minor arc in Theorem \ref{thm4}.}\label{sec:minorarcthm4}

Here we use \eqref{eq1.6}, so in the present section we show that (with $S^+$ to be specified below)
 \begin{equation}\label{eq7.1}
\int_{X^{2-c}}^K S(x) G(x) dx \ll X^{5-c-3\eta} 
 \end{equation}
where $G(x)$ is either $S^3(x) S^+(x) \Phi(x) e(-Rx)$ or $S^2(x) S^+(x)^2 \Phi(x) e(-Rx)$.

Let us write $\|\ldots\|$ for sup norm on $[X^{2-c}, K]$. It suffices to show that
 \begin{equation}\label{eq7.2}
A(x) \ll X^{5-2c-14\eta} + X^{1-c} x^{-1} \quad (0 < x \le 2X^{2\eta})
 \end{equation}
and that
 \begin{align}
\|S\|_\infty &\ll X^{79/80 + 2\eta};\label{eq7.3}\\[2mm]
\|S^+\|_\infty &\ll X^{0.968 + 2\eta}.\label{eq7.4}
 \end{align}
Using the bounds in Lemma \ref{lem13} (ii), together with Lemma \ref{lem12} (ii),
 \begin{align*}
&\left|\int_{X^{2-c}}^K S(x) G(x) dx\right|^2\\[2mm]
&\hskip .5in \ll \mc L^4 X^{2-c}(\|S\|_\infty^2\, \|S^+\|_\infty^2 + \|S^+\|_\infty^4)X^{2+3\eta}\\[2mm]
&\hskip 1.25in + X^{5-c-14\eta} \, X\, \mc L^2 X^{4+6\eta}\\[2mm]
&\hskip .5in \ll X^{2-c+ 2\left(\frac{79}{80} + 0.968\right) + 2 + 4\eta} + X^{10-2c-7\eta}\\[2mm]
&\hskip .5in \ll X^{10-2c-6\eta}
 \end{align*}
as required for \eqref{eq7.1}.

We turn to \eqref{eq7.2}. This is obtained from Lemma \ref{lem7} with $xX^c$, $X$ in place of $T$, $N$. The Kusmin-Landau theorem gives
 \[A(x) \ll X^{1-c} |x|^{-1}\]
unless $X^{c-1}x \gg 1$, which we now assume. If
 \[X \le (x X^c)^{25/42}\]
we can use Lemma \ref{lem7}, since
 \[(xX^c)^{17/42} \ll X^{2.09 \times 17/42}
 \ll X^{1-\eta};\]
we obtain
 \[A(x) \ll (X^{c+2\eta})^{\frac{13}{84} + \eta} 
 X^{\frac 12} \ll X^{5-2c-14\eta}\]
since $13c + 42 < 420-168c$ (this inequality determines the range of $c$ in the theorem).

In the remaining case $X > (xX^c)^{25/42}$, we apply Lemma \ref{lem5} (i) with $\ell = 1$:
 \[A(x) \ll (xX^c)^{1/6} X^{\frac 12} \ll
 X^{7/25 + \frac 12 + \eta} = X^{0.78+\eta}\]
which suffices for \eqref{eq7.2}.

Turning to \eqref{eq7.3}, we first show that Type II sums are $O(X^{79/80 + \eta})$ whenever
 \[X^{1/40} \ll N \ll X^{\frac 12}\]
(and hence whenever $X^{1/40} \ll N \ll X^{39/40}$). Proceeding as in \eqref{eq3.9}, we need to show
 \[\sum_{\substack{m\sim M\\
 X < mn \le 2X}} e(xm^c((n+q)^c-n^c))
 \ll MX^{-\frac 1{40} + \eta}\]
whenever $X^{39/40} \gg M \gg X^{1/2}$; here $Q = X^{1/40}$, $1 \le q \le Q$. We apply Lemma \ref{lem8} with $k=5$; here
 \[f^{(5)}(x) \asymp X^c q N^{-1} M^{-5}.\]
For the second term on the right-hand side in \eqref{eq2.1} we have the bound
 \[\ll M^{\frac{19}{20} + \eta} \ll MX^{-\frac 1{40} + \eta}\]
since $M \gg X^{1/2}$. We can absorb the first term into the second:
 \[x\, X^c q \, N^{-1}\, M^{-5} \ll M^{-1},\]
because $M^4N \gg X^{5/2}$. For. the third term, we have
 \[M^{1+\eta} (x\, X^c q \, N^{-1}M^{-5})^{-\frac 1{50}}
 M^{-1/10} \ll MX^{-\frac 1{40}}\]
since $x\, X^cq N^{-1} \gg X^{3/2}$.

We claim that Type I sums are $\ll X^{\frac{79}{80}}$ whenever $N \ge X^{39/40}$. Using a familiar estimate,
 \begin{align*}
S_I &\ll M(x X^c)^{1/14} N^{10/14}\\[2mm]
&\ll X^{1 + \frac{2.1}{14}} N^{-\frac 27} \ll X^{79/80}. 
 \end{align*}
It is now clear from Lemma \ref{lem14} that \eqref{eq7.3} holds.

As for \eqref{eq7.4}, we take
 \begin{equation}\label{eq7.5}
S^+(x) = \sum_{\frac X8 < n \le X} \rho(n, X^{0.064}) - \sum_{X^{0.064} \le p_1 \le X^{0.317}} \ \sum_{\substack{\frac X8 < p_1n\le X\\
(n,P(p_1))=1}} e(x(p_1n)^c).
 \end{equation}
Using Buchstab's identity, we have $\rho^+ \ge \rho$ since
 \begin{equation}\label{eq7.6}
S^+(x) = \sum_{\frac X8 < n \le X} \rho(n) e(xn^c) + \sum_{X^{0.317} < p_1 < (3X)^{\frac 12}} \sum_{\substack{\frac X8 < p_1n\le X\\
(n,P(p_1))=1}} e(x(p_1n)^c). 
 \end{equation}
We show that \eqref{eq7.4} holds using Lemmas \ref{lem15}, \ref{lem16}. We claim first that
 \[S_I(x) \ll X^{0.968}\]
for $N \gg X^{7/10}$. To see this,
 \begin{align*}
S_I(x) &\ll M(xX^c)^{1/14} N^{10/14}\\[2mm]
&\ll X^{1+2.09/14}(X^{7/10})^{-2/7} \ll X^{0.95}. 
 \end{align*}
Next, we claim that
 \begin{gather}
S_{II}(x) \ll X^{0.968}\label{eq7.7}\\
\intertext{for}
X^{0.064} \ll N \ll X^{0.317}.\notag
 \end{gather}
By a familiar argument, we need to show that for $1 \le q \le Q : = X^{0.064}$, $n \sim N$ we have
 \[S_* : = \sum_{m\in I} e\, (x((n+q)^c-n^c)
 m^c) \ll MX^{-0.064}\]
($I$ is a subinterval of $(M,2M]$). We have
 \begin{align*}
S_* M^{-1}X^{0.064} &\ll (xqN^{-1}X^c)^{1/14} M^{10/14}
 M^{-1}X^{0.064}\\[2mm]
 &\ll X^{(0.064\times 15 + 2.0884)/14}X^{-(1+3\times 0.683)/14} \ll 1,
 \end{align*}
proving \eqref{eq7.7}.

We may now apply Lemma \ref{lem15} with $\a = 0.064$, $\b = 0.064$, $M = X^{0.3}$, $R = S = 1$, $w(n) = e(xn^c)$. We obtain the desired bound
 \[\sum_{n\sim X} \rho(n, X^{0.064}) \, e(xn^c) \ll
 X^{0.968 + \eta},\]
while the sum
 \[\sum_{X^{0.064}\, \le\, p_1\, \le \, X^{0.317}} \
 \sum_{\substack{
 \frac X8 < p_1n \le X\\
 (n,\, P(p_1))=1}} e(x(p_1n)^c)\]
satisfies the same bound from Lemma \ref{lem16}. This completes the discussion of the minor arc.
 \bigskip
 
 \section{Minor arc in Theorem \ref{thm6}.}\label{sec:minorarcthm5_2} 
 
We shall show, for suitably chosen $T^+(x)$, that
 \begin{equation}\label{eq8.1}
\int_{X^{2-c}}^{1/2} T(x)\, G(x) dx \ll X^{5-c-3\eta} 
 \end{equation}
where $G(x)$ is either of $T(x)^3T^+(x) \Phi(x)e(-Rx)$ or $T(x) T^+(x)^2 \Phi(x) e(-Rx)$.

In Lemma \ref{lem13} (ii) we take
 \[\mc S(x) = T(x), \quad \mc J(x) = B(x)\]
and $G$ as above. Suppose for the moment that
 \begin{equation}\label{eq8.2}
B(x) \ll X^{5-2c-20\eta} + \mc L X^{1-c} x^{-1} \quad (0 < x \le 2K) 
 \end{equation}
and that, with
 \[T^+(x) = \sum_{\frac X8 < n \le X} \rho(n, X^{0.064})
 e(x[n^c]) - \sum_{X^{0.064}\le p_1 \le X^{0.317}} 
 \sum_{\substack{
 \frac X8 < p_1n \le X\\
 (n, \, P(p_1))=1}} e(x[(p_1n)^c],\]
(as in Section \ref{sec:minorarcthm4}, mutatis mutandis), we have
 \begin{align}
&\|T\|_\infty \ll X^{\frac{79}{80} + 3\eta},\label{eq8.3}\\[2mm]
&\|T^+\|_\infty \ll X^{0.97095 + 3\eta}.\label{eq8.4} 
 \end{align}
This implies (using Lemma \ref{lem13} and Lemma \ref{lem12} (ii)) that
\newpage
 \begin{align*}
\left(\int_{X^{2-c}}^{1/2} T(x) G(x) dx\right)^2 &\ll X^{2-c+7\eta + \frac{79}{40} + 1.9419+ 2}\\[2mm]
&\quad + X^{5-2c - 20\eta + 5 + 8\eta} \ll X^{10-2c - 6\eta}
 \end{align*}
as required for \eqref{eq8.1}.

Let $H = X^{2c-4+30\eta}$. In order to prove \eqref{eq8.2} it suffices to obtain
 \begin{equation}\label{eq8.5}
\sum_{0 \le h \le H} \min\left(1, \frac 1h\right) \Bigg|\sum_{\frac X8 < n \le X} e((h + \g)n^c)\Bigg| \ll X^{5-2c-20\eta} 
 \end{equation}
for $\g\in\{x, -x, 0\}$, and
 \begin{equation}\label{eq8.6}
\sum_{h > H} \ \frac H{h^2} \Bigg|\sum_{\frac X8 < n \le X} e(hn^c)\Bigg| \ll X^{5-2c-20\eta}. 
 \end{equation}

For the contribution from $h=0$ in \eqref{eq8.5}, we use the analysis leading to \eqref{eq7.2}.

For the contribution from $h\sim H_1$ in \eqref{eq8.5}, we apply Lemma \ref{lem6}, with $T \asymp H_1 X^c$ and $N=X$, with $y_h = \frac{h+\g}{H_1 + \g}$. The condition $T^{1/3} \le X \le T^{1/2}$ is obviously satisfied. We must show that
 \begin{gather}
H_1^{\frac{319}{345}} X^{\frac{449}{690}} (H_1X^c)^{\frac{63}{690} + \eta} \ll H_1X^{5-c-20\eta},\label{eq8.7}\\
\intertext{and that}
X^{\frac 12}(H_1X^c)^{\frac{141}{950}+\eta} \ll X^{5-c-20\eta}.\label{eq8.8}
 \end{gather}
The worst case in \eqref{eq8.7} is clearly $H_1 = H$. We verify that
 \[\frac{11}{690}\, (2c-4) + \frac{449}{690}
 + \frac{63c}{690} < 5-2c,\]
which reduces to $c < \frac{609}{293}$. (This determines the range of $c$ in Theorem \ref{thm5}). In \eqref{eq8.8} we require
 \[\frac 12 + (3c-4) \, \frac{141}{950} < 5-2c,\]
which holds for $c < \frac{609}{293}$ with a little to spare.

The contribution to the left-hand side of \eqref{eq8.6} from $h \sim H_1$, $H \le H_1 < \frac 12\, X^{3-c}$, can be handled using \eqref{eq8.7}, \eqref{eq8.8} since we have
 \[(2H_1 X^c)^{\frac 13} \le X \le (H_1X^c)^{\frac 12}.\]
The additional factor $H/H_1$ arising from $H/h^2$ leads to a negative exponent of $H_1$ in using \eqref{eq8.7}, \eqref{eq8.8}.

For $H_1 \ge \frac 12\, X^{3-c}$, we use Lemma \ref{lem5} (i) with $\ell=2$: we need to verify that
 \[\frac H{H_1}\, (H_1X^c)^{\frac 1{14}}
 X^{\frac{10}{14}} < X^{5-2c-20\eta}.\]
The worst case is $H_1 = \frac 12\, X^{3-c}$. Here
 \[H^{14} H_1X^c X^{10} < H_1^{14}
 X^{70-28c-300\eta}\]
since $70c < 155$. This completes the discussion of \eqref{eq8.2}.

We also treat $T(x)$ and $T^+(x)$ using Lemma \ref{lem1}. For $T(x)$, we choose $H = X^{1/80}$. Now for \eqref{eq8.3} it suffices to show that for $X^{1/40} \ll N \ll X^{1/2}$, $X < X' \le 2X$, $1 \le h \le H$, we have
 \begin{equation}\label{eq8.9}
S_{II} : = \underset{\frac X8 < mn \le X}{\sum_{m\sim M} a_m \sum_{n\sim N}} b_n e((h+\g)(mn)^c) \ll X^{79/80 + 2\eta}
 \end{equation}
for $\g \in \{x, -x, 0\}$; \textit{and}, with the same ranges of $h$, $x$, $\g$ and $N \gg X^{9/10}$,
 \begin{equation}\label{eq8.10}
S_I : = \underset{\frac X8 < mn \le X}{\sum_{m\sim M} \ \sum_{n\sim N}} e((h+\g)m^cn^c) \ll X^{79/80 + 2\eta}.
 \end{equation}
\bigg(We already have a satisfactory bound for the sum
 \[\sum_{h > H}\ \frac H{h^2}\
 \Bigg|\sum_{\frac X8 < n \le X} e(hn^c)\Bigg|.\Bigg)\]

We begin with \eqref{eq8.9}. By a familiar argument, we must show that, for $n\sim N$, $q \le Q: = x^{1/40}$, we have
 \[\sum_{\substack{
 m\sim M\\[1mm]
 \frac X8 < mn \le X'}} e((h+\g) m^c((n+q)^c-n^c))
 \ll MQ^{-1}X^\eta.\]
We apply Lemma \ref{lem8} with $k=5$,
 \[\l_5 = (h+\g)q X^c N^{-1} M^{-5}.\]
Note that
 \[\l_5 \ll M^{-1}\]
since $NM^4 \gg X^{5/2}$. As for the second term in the bound in \eqref{eq2.1}, it is
 \[\ll M^{19/20 + \eta} \ll MX^{-\frac 1{40}+\eta}.\]
For the third term,
 \[M^{1-1/10} ((h+\g) q X^c N^{-1}M^{-5})^{-1/50}
 \ll MX^{-1/40}\]
since $(h+\g) X^c \gg X^2$ and $X^2N^{-1}\gg X^{3/2}$. This proves \eqref{eq8.9}.

Now we readily verify \eqref{eq8.10} on bounding $S_I$ by
 \begin{align*}
&\ll M(HX^{c+3\eta})^{1/14} N^{10/14}\\[2mm]
&\ll X^{(1/80+c+3\eta)/14} (X^{9/10})^{-2/7} X\ll X^{0.9}.
 \end{align*}
This establishes \eqref{eq8.10}. 

For $T^+(x)$ we proceed similarly, except that we now take $H = X^{0.02905}$, and instead of the range $[X^{\frac 1{40}}, X^{\frac 12}]$, we have
 \[Q := X^{0.0581} \ll N \ll X^{0.317}.\]
The discussion of \eqref{eq8.9} goes as before, and it only remains to obtain the bound corresponding to \eqref{eq8.10}. It suffices to show that, for $N \gg X^{9/10}$,
 \[(HX^c)^{\frac 1{14}} N^{10/14} \ll NX^{-0.06},\]
which is true with something to spare. This completes the proof of \eqref{eq8.4} and the treatment of the minor arc.
 \bigskip

 \section{Major arc in Theorems \ref{thm1}--\ref{thm6}.}\label{sec:minorarcthms1_6}

The arguments in the present section are adapted from \cite{kum,bakwein,laptol}. We begin with a number of lemmas. Let
\newpage
 \begin{gather*}
v_1(X, x) = \int_0^X e(x\g^c)d\g,\\[2mm]
v(X, x) = \sum_{1 \le m \le X} \frac 1c\, m^{1/c-1} e(xm).
 \end{gather*}
Proofs of Lemmas \ref{lem18} and \ref{lem19} (ii) can be found in Vaughan \cite[Sections 2.4, 2.5]{vau} with the unimportant difference that $c \in\mb N$ in \cite{vau}, while Lemma \ref{lem19} (i) follows from \cite[Lemma 3.1]{grakol}.

 \begin{lem}\label{lem18}
We have
 \[v(X, x) = v_1(X,x) + O(1 + X^c|x|).\]
 \end{lem}

 \begin{lem}\label{lem19}\
 \begin{enumerate}
\item[(i)] We have
 \[v_1(X,x) - v_1(X/8, x) \ll (|x|\, X^{c-1})^{-1}.\]
 
\item[(ii)] For $|x| \le 1/2$, we have
 \[v(X,x) \ll |x|^{-1/c}.\]
 \end{enumerate}
 \end{lem}

 \begin{lem}\label{lem20}
For $2 \le s \le 5$ and $r$ large, let $X = r^{1/c}$. We have
 \[L_s : = \int_{-1/2}^{1/2} \left(v(X^c, x) - 
 v\left(\left(\frac X8\right)^c, x\right)\right)^s 
 e(-rx)dx \gg r^{s/c-1}.\] 
 \end{lem}
 
 \begin{proof}
The integral is
 \begin{align*}
\frac 1{c^s} \, &\sum_{\left(\frac X8\right)^c < m_j \le X^c\
 (1 \le j \le s)} (m_1\ldots m_s)^{\frac 1c-1}
 \int_{-\frac 12}^{\frac 12} e(x(m_1+\cdots + m_s-r))dx\\[2mm]
&= \frac 1{c^s} \sum_{\substack{\left(\frac X8\right)^c < m_j \le X^c\\
m_1 + \cdots + m_s =r}} (m_1\ldots m_s)^{\frac 1c- 1}\\[2mm]
&\gg r^{\frac sc - s} \ \sum_{\frac r8 < m_j \le \frac r7\, (1\le j \le s-1)} 1.
 \end{align*} 
The last step is valid because for each choice of $m_1, \ldots, m_{s-1}$ in the last sum we have
 \[m_1 + \cdots + m_{s-1} \le \frac r7\, (s-1) \ , \
 r \ge r - (m_1 + \cdots + m_{s-1}) > \frac r8,\]
hence $r - (m_1 + \cdots + m_{s-1}) = m_s$ with $\frac r{8^c} < m_s \le r$. Now the desired lower bound follows at once.
 \end{proof}
 
 \begin{lem}\label{lem21}
For $2 \le s \le 5$, we have
 \[H_s = \int_{X/8}^X \cdots \int_{X/8}^X 
 \phi(t_1^c + \cdots + t_s^c - R) dt_1 \ldots dt_s \gg
 X^{s-c-c\eta}.\]
 \end{lem}
 
 \begin{proof}
One verifies easily that for each choice of $t_1,\ldots, t_{s-1}$ from $\left[\frac X8, \frac X7\right]$, there is an interval of $t_s$ in $[X, 2X]$ of length $\gg X^{1-c-c\eta}$ on which
 \[\phi(t_1^c + \cdots + t_s^c - R) = 1.\qedhere\]
 \end{proof}

A \textit{polytope} means a bounded intersection of half-spaces in $\mb R^j$. The polytope $P_j$ is defined by
 \begin{align*}
P_j = \bigg\{(y_1,\ldots,y_j) : \b &\le y_j < y_{j-1} < \cdots < y_1,\\
 &y_1 + \cdots + y_{j-1} + 2y_j \le
 1 + \frac{\log 3}{\mc L}\bigg\},
 \end{align*}
where $\b = 8/75$. In writing sums containing $p_1, \ldots, p_j$, it is convenient to set
 \begin{gather*}
\bs\a_j = (\a_1, \ldots, \a_j) : = \frac 1{\mc L}\, (\log p_1, \ldots, \log p_j),\\
f_1(\a_1) = \a_1^{-2}, \ f_j(\bs\a_j) = (\a_1 \ldots \a_{j-1})^{-1} \a_j^{-2} \ (j\ge 2),
 \end{gather*}
$\pi_j = p_1 \cdots p_j$, $\pi_j' = p_1'\cdots p_j'$. Let $\omega(\ldots)$ denote Buchstab's function.

 \begin{lem}\label{lem22}
Let $E$ be a polytope, $E \subseteq P_j$. Let $j+1\le k \le 9$.

(i) Let
 \[S_k(E) = \sum_{\bs\a_j \in E} \ \sum_{\substack{
 p_j \le p_{j+1} \le \cdots \le p_{k-1}\\[.5mm]
 \pi_{k-1} p_{k-1}\le X}} \frac 1{\pi_{k-1}}\,.\]
Then
 \[S_k(E) \ll 1.\]

(ii) Let
 \[S_k^*(E) = \sum_{\bs\a_j\in E} \ \sum_{\substack{
 p_j \le p_{j+1} \le \cdots \le p_{k-1}\\[.5mm]
 \frac X8 < \pi_{k-1}p_{k-1} \le X}} \frac 1{\pi_{k-1}}\, .\]
Then
 \[S_k^*(E) \ll \mc L^{-1}.\]
 \end{lem}
 
 \begin{proof}
Mertens' formula \cite[Chapter 7]{brufou} implies
 \[\sum_{A < p \le B} \ \frac 1p = \log\, 
 \frac{\log B}{\log A} + O(\mc L^{-1}) \ \ 
 (X^\b \le A < B \le 2X).\]

Now
 \begin{align*}
S_k(E) &\le \sum_{X^\b \le p_1 \le X^{1-\b}} \frac 1{p_1} \cdots \sum_{X^\b \le p_{k-1} \le X^{1-\b}} \frac 1{p_{k-1}}\\
&\le \left(\log \, \frac{\log X^{1-\b}}{\log X^\b} + O(\mc L^{-1})\right)^{k-1} \ll 1.
 \end{align*}
In $S_k^*(E)$ we replace the factor $\sum\limits_{X^\b \le p_{k-1} \le X^{1-\b}} \frac 1{p_{k-1}}$ by
 \[\sum_{\frac X{8\pi_{k-1}} < p_{k-1}< \frac X{\pi_{k-1}}}\]
and use
 \[\log\left(\frac{\log\, \frac X{\pi_{k-1}}}{
 \log\, \frac X{8\pi_{k-1}}}\right) = \log\left(1 +
 \frac{\log 8}{\log\, \frac X{8\pi_{k-1}}}\right)
 \ll \mc L^{-1}.\qedhere\]
 \end{proof}
 
 \begin{lem}\label{lem23}
(i) Let $2 \le Z < Z' \le 2Z$. We have, for $0 < y < X^{1-c-2\eta}$,
 \[\sum_{Z \le p < Z'} e(p^c y) = \int_Z^{Z'}
 \frac{e(u^cy)}{\log u}\, du + O(Z \exp (-C(\log Z)^{1/4})).\] 
(ii) Let $E$ be a polytope, $E \subseteq P_j$. Let $j + 1 \le k\le 9$. Then for $0 < x \le \tau$, $X < X' \le 2X$, we have
 \newpage
 
 \begin{align*}
\underset{\frac X8 < p_1\cdots p_k \le X}{\sum_{\bs \a_j \in E} \ \sum_{p_j\le p_{j+1} \le \cdots \le p_k}} &e(\pi_k^c x) = \underset{\frac X8 < p_1\cdots p_k \le X}{\sum_{\bs \a_j \in E} \ \sum_{p_j\le p_{j+1} \le \cdots \le p_k}} \frac 1{\pi_{k-1}}\\[2mm]
&\int_{\max\left(\pi_{k-1}p_{k-1}, \frac X8\right)}^X \frac{e(t^cx)}{\log(t/\pi_{k-1})}\, dt + O(X \exp (-C\mc L^{1/4})).
 \end{align*}
(iii) The assertions of (i), (ii) remain valid if $e(p^cx)$, $e(\pi_k^c x)$ are replaced respectively by $e([p^c]x)$, $e([\pi_k^c]x)$.
 \end{lem}
 
 \begin{proof}
(i), (ii) are slight variants of \cite[Lemma 24]{bakwein} and \cite[Lemma 21]{bakwein2} respectively. For (iii) we note that,  when $a_n\ll 1$, $x \ll \tau$, 
 \begin{align*}
\sum_{n\le 2X} a_n e(n^cx) &- \sum_{n\le 2X} a_ne([n^c]x)\\
&\ll X\tau \ll X^{1-c/2}.\qedhere 
 \end{align*}
 \end{proof}

 \begin{lem}\label{lem24}
Let $E$ be a polytope, $E \subseteq P_j$. Let
 \[f(E; X) = \sum_{\bs\a_j \in E} \ \sum_{j+1 \le k \le 9} \
 \sum_{\substack{p_j \le p_{j+1} \le \cdots \le p_{k-1}\\[.5mm]
 \pi_{k-1} p_{k-1} \le X}} \ 
 \frac 1{\pi_{k-1}\log(X/\pi_{k-1})}.\]
As $X\to \infty$, we have
 \[f(E; X) = (1 + o(1)) \, \frac 1{\mc L}\int_E
 f_j(\bs z_j) \omega\left(\frac{1-z_1-\cdots-z_j}{z_j}\right)
 dz_1\ldots dz_j.\]
 \end{lem}
 
 \begin{proof}
This is a slight variant of \cite[Lemma 20]{bakwein2}.
 \end{proof}

We now discuss the major arc for Theorems \ref{thm1}--\ref{thm6}, and complete the proofs of the theorems.
 \medskip

(i) \textbf{Theorem \ref{thm3}}. We easily verify that for functions $f_j$ $(1 \le j \le s, 3 \le s \le 5)$ and $g$ having
 \begin{gather*}
\sup_{x\in [-\tau,\tau]} |f_j| \ll X, \ \int_{-\tau}^\tau |f_j|^2 dx \ll X^{2-c}\mc L,\\
\sup_{x\in [-\tau,\tau]} |f(x) - g(x)| \ll X \exp (-C\mc L^{1/4}),
 \end{gather*}
we have
 \begin{align}
\int_{-\tau}^\tau g(x) f_2(x) &\ldots f_s(x) \Phi(x) e(-Rx)dx\label{eq9.1}\\
&- \int_{-\tau}^\tau f_1(x) f_2(x) \ldots f_s(x) \Phi(x) e(-Rx)dx\notag\\
&\hskip .75in \ll X^{s-c-c\eta} \exp(-C\mc L^{1/4}).\notag
 \end{align}
Thus in view of Lemma \ref{lem23} (i), we can replace $\int_{-\tau}^\tau S_1(x)^4\Phi(x)dx$ by $\int_{-\tau}^\tau I(x)^4 \Phi(x) e(-Rx)dx$, replacing factors one at a time, with error $\ll X^{4-c-c\eta}\exp(-C\mc L^{1/4})$. Now we extend the integral to $\mb R$ with total error $\ll X^{4-c-c\eta}\exp(-C\mc L^{1/4})$ using Lemma \ref{lem19} (i) and the case $s=4$ of
 \begin{align}
\int_\tau^\infty |x|^{-s} X^{s-cs} |\Phi(x)|dx &\ll (X^{-c+8\eta})^{-s+1} X^{s-cs-c\eta}\label{eq9.2}\\
&= X^{s-c-c\eta-8(s-1)\eta}.\notag
 \end{align}
We find using \eqref{eq1.4} and the bound \eqref{eq4.1} that
 \begin{equation}\label{eq9.3}
\mc L^4 \mc A_4(R) \gg \int_{-\infty}^\infty I(x)^4 \Phi(x) e(-Rx)dx + O(X^{4-c-c\eta} \exp (-C\mc L^{1/4})).
 \end{equation}

The integral here is
 \begin{align}
\int_X^{2X} \int_X^{2X} \int_X^{2X} &\int_X^{2X} \phi(t_1^c + \cdots + t_4^c - R) dt_1\ldots dt_4\label{eq9.4}\\[2mm]
&\gg X^{4-c-c\eta}\notag
 \end{align}
by Lemma \ref{lem21}. This yields Theorem \ref{thm3} at once.
 \medskip

(ii) \textbf{Theorem \ref{thm1}}. If $f_1$, $f_2$, $g$ satisfy $f_j \ll X$, $\int_{-\tau}^\tau f_j^2 \ll X^{2-c}\mc L$,
 \[\sup_{x\in [-\tau,\tau]} |f_1(x) -g(x)| 
 \ll X \exp (-C\mc L^{1/4}),\]
then the integral
 \[\int_{-\tau}^\tau (f_1(x) - g(x)) f_2(x)
 \Phi(x) e(-Rx)dx\]
is of the form $\widehat E(R)$ where
 \[E(y) = \begin{cases}
 (f_1(y) - g(y)) f_2(y) \Phi(y) & (y \in [-\tau,\tau])\\
 0 & \text{otherwise.} \end{cases}\]
By Parseval's formula
 \begin{align*}
\int_V^{2V} |\widehat E(R)|^2 dR &< \int_{\mb R} |E(y)|^2 dy\\[2mm]
&= \int_{-\tau}^\tau (f_1(y) - g(y))^2 f_2^2(y) \Phi^2(y)dy\\[2mm]
&\ll X^{2-c-2c\eta} X^2 \exp (-C \mc L^{1/4})\\[2mm]
&\ll X^{4-c-c\eta} \exp (-C\mc L^{1/4}).
 \end{align*}
Thus in two steps we can replace $\int_{-\tau}^\tau S_1(x)^2 \Phi(x) e(-Rx)dx$ by $\int_{-\tau}^\tau I(x)^2$ $\Phi(x) e(-Rx)dx$ with an error that is acceptable for Theorem \ref{thm1}. (Compare the discussion of $E_0(x)$ in Section \ref{sec:minorarcthms12}.) Similarly in replacing

 \[\int_{-\tau}^\tau I(x)^2 \Phi(x) e(-Rx)dx \quad \text{by}
 \quad \int_{\mb R} I(x)^2 \Phi(x) e(-Rx)dx\]
we incur an error $E_1(x)$ with
 \[\int_V^{2V} |\widehat E_1(R)|^2 dR <
 \int_\tau^\infty |I(y)|^4 |\Phi(y)|^2dy\]
which from \eqref{eq9.2} is $\ll X^{4-c-c\eta-24\eta}$. Now we easily adapt the argument leading to \eqref{eq9.3} to obtain
 \[\mc L^2\mc A_2(R) \gg X^{2-c-c\eta}\]
except for a set of $R$ in $[V,2V]$ whose measure is $\ll V \exp (-C(\log V)^{1/4})$, proving Theorem \ref{thm1}.
 \medskip

(iii) \textbf{Theorem \ref{thm5}}. In the minor arc for Theorem \ref{thm5}, $T_1(x) - S_1(x) = O(X\tau)$. We may replace $\int_{-\tau}^\tau T_1(x)^3 \Phi(x) e(-Rx)dx$ by $\int_{-\tau}^\tau S_1(x)^3 \Phi(x) e(-Rx)dx$ with error $O(X^{3-c-3\eta})$ since 
 \begin{equation}\label{eq9.5}
T_1^3 - S_1^3 \ll X^2 X\tau \ , \ \int_{-\tau}^\tau |T_1^3 - S_1^3| dx \ll X^{3-2c + 16\eta}.
 \end{equation}
Now we replace $\int_{-\tau}^\tau S_1(x)^3 \Phi(x) e(-Rx)dx$ by $\int_{-\tau}^\tau J(x)^3 \Phi(x) e(-Rx)dx$ with error $O(X^{3-c-c\eta} \exp (-C\mc L^{1/4}))$ using Lemmas \ref{lem18} and \ref{lem23} (i). We then extend the integral to $\left[-\frac 12, \frac 12\right]$ using Lemma \ref{lem19} (ii); here we note that
 \begin{equation}\label{eq9.6}
\int_\tau^{1/2} x^{-3/c} dx < (X^{-c + 8\eta})^{-\frac 3c + 1} < X^{3-c-8\eta}. 
 \end{equation}
Now we can complete the proof of Theorem \ref{thm3} by drawing on Lemma \ref{lem20} together with \eqref{eq1.5} and the result of Section \ref{sec:minorarcthm5}.
 \medskip

(iv) \textbf{Theorem \ref{thm2}}. We consider the sum $S^+$ on the major arc. We decompose $S^+$ into $S$ plus $O(1)$ sums of the type
 \begin{equation}\label{eq9.7}
U_k(E, x) := \underset{\frac X8 < p_1\cdots p_k \le X}{\sum_{\bs \a_j \in E} \ \sum_{p_j \le p_{j+1} \le \cdots \le p_k}} e(x\pi_k^c) 
 \end{equation}
where $1 \le j \le 3$ and $E$ is a polytope, $E \subseteq P_j$. Recalling Lemma \ref{lem23} (ii), we replace $U_k(E, x)$ by
 \begin{align}
V_k(E, x) &:= \underset{\pi_{k-1} p_{k-1} \le X}{\sum_{\bs \a_j \in E} \ \sum_{p_j \le p_{j+1} \le \cdots \le p_{k-1}}} \, \frac 1{\pi_{k-1} \log(X/\pi_{k-1})}\label{eq9.8}\\[2mm]
&\qquad \left\{v_1(X, x) - v_1\left(\max\left(\pi_{k-1} p_{k-1}, \frac X8\right)\right)\right\}\notag
 \end{align}
with error $O(X\mc L^{-1})$. (We include $\mc L^{-1}I(x)$ as a term $V_k(E,x)$ for convenience.)

By an obvious variant of the argument leading to \eqref{eq9.1}, we can replace
 \[\int_{\tau}^\tau S^2(x) S^+(x) \Phi(x) e(-Rx)dx\]
by
 \begin{equation}\label{eq9.9}
\sum_{(k, E)} \ \frac 1{\mc L^2} \ \int_{-\tau}^\tau I(x)^2 V_k(E,x) \Phi(x) e(-Rx)dx,
 \end{equation}
and replace
 \[\int_{\tau}^\tau S(x) S^+(x)^2 \Phi(x) e(-Rx)dx\]
by
 \begin{equation}\label{eq9.10}
\sum_{(k, E)} \ \sum_{(k',E')} \ \frac 1{\mc L} \ \int_{-\tau}^\tau I(x) V_k(E,x) V_{k'} (E',x) \Phi(x) e(-Rx)dx
 \end{equation}
with error $O(X^{3-c-c\eta} \mc L^{-4})$. We can extend the integrals in \eqref{eq9.9} and \eqref{eq9.10} to $\mb R$ with error $O(X^{3-c-c\eta} \mc L^{-4})$ using Lemma \ref{lem19} (i) and \eqref{eq9.2}.

We now observe that (omitting regions of summation)
 \begin{align*}
&\frac 1{\mc L} \int_{-\infty}^\infty I(x) V_k(E,x) V_{k'}(E', x) \Phi(x) e(-Rx)dx\\[2mm]
&= \sum_{p_1,\ldots,p_{k-1}} \ \sum_{p_1', \ldots, p_{\ell-1}'} \ \frac 1{\pi_{k-1}\, \pi_{\ell-1}'} \frac 1{(\log X)(\log X/\pi_{k-1})(\log X/\pi_{\ell-1}')}\\[2mm]
&\qquad \int_{X_3}^X \int_{X_2}^X \int_{X/8}^X \int_{-\infty}^\infty e(x(t_1^c + t_2^c + t_3^c - R)) \Phi(x)dx\, dt_1\, dt_2\, dt_3,
 \end{align*}
where $X_3 = \max\left(\pi_{k-1}p_{k-1}, \frac X8\right)$, $X_2 = \max\left(\pi_{\ell-1}' p_{\ell-1}', \frac X8\right)$.

We rewrite the last expression as
 \begin{align*}
&\sum_{p_1,\ldots, p_{k-1}} \ \sum_{p_1',\ldots, p_{\ell-1}'}
 \ \frac 1{\pi_{k-1}\pi_{\ell-1}'} \  \frac 1{(\log X)(\log X/\pi_{k-1})(\log X/\pi_{\ell-1}')}\\[2mm]
&\qquad \int_{X_3}^{2X} \int_{X_2}^{2X} \int_X^{2X} \phi(t_1^c + t_2^c + t_3^c - R) dt_1\, dt_2\, dt_3.
 \end{align*}
We replace $X_2$, $X_3$ by $X/8$, inducing an error that is $O(\mc L^{-4}H_3)$ by Lemma \ref{lem22} (ii). This produces the quantity
 \[\frac{H_3}{\mc L} \Bigg(\sum_{p_1,\ldots, p_{k-1}} \
 \frac 1{\pi_{k-1}\log\, \frac X{\pi_{k-1}}}\Bigg) ,\Bigg(
 \sum_{p_1', \ldots, p_{k-1}'} \ 
 \frac 1{\pi_{\ell-1}' \log X/\pi_{\ell-1}'}\Bigg),\]
which can  be calculated to within a factor $1 + o(1)$ using Lemma \ref{lem24}. Following a similar argument with the integrals in \eqref{eq9.9}, we arrive at representations, to within a factor $1 + o(1)$, of
 \[\int_{-\tau}^\tau S^2(x) S^+(x) \Phi(x) e(-Rx)dx, \
 \int_{-\tau}^\tau S(x) S^+(x)^2 \Phi(x) e(-Rx)dx\]
of the respective forms
 \[\frac{u^+H_3}{\mc L^3} \ , \ 
 \frac{(u^+)^2 H_3}{\mc L^3}\, ,\]
where (recalling \eqref{eq5.15}), $u^+ = 1 + d_1 + d_2 + d_3 + d_4$,
 \begin{gather*}
d_1 = \int_{11/25}^{1/2}\ \frac{dx}{x(1-x)} \ , \ d_3 = \int_{29/105}^{1/3}\ \frac{dx}{x(1-x)}\, ,\\[2mm]
d_2 = \int_{11/75}^{1/5} \int_{\max\left(\frac{29}{210}, \frac 12 \left(\frac{11}{25} - x - y\right)\right)}^{\min\left(x, \frac 13 - x\right)} \int_{\max\left(\frac{11}{25} - x - y, \frac{29}{105}-y\right)}^y
\frac 1{xyz^2}\\[2mm]
\hskip 1.75in\omega\left(\frac{1-x-y-z}z\right) dzdydx,\\
d_4 = \int_{29/105}^{7/25} \int_x^{14/25-x} \frac{dydx}{xy(1-x-y)}\, .
 \end{gather*}
Taking into account \eqref{eq1.6} and the result of Section \ref{sec:minorarcthm3}, we find that
 \begin{equation}\label{eq9.11}
\mc A_3(R) \gg (1 + o(1)) (2u^+ - (u^+)^2)\, \frac{H_3}{\mc L^3}\, .
 \end{equation}
Using a computer calculation for $d_3$ and $d_4$, we find that
 \[d_1 < 0.242, \, d_2 < 0.016, \, d_3 < 0.272,\,
 d_4 < 0.001.\]
Thus $u^+ \in(1,2)$, and Theorem \ref{thm2} follows from \eqref{eq9.11}.
 \bigskip

(v) \textbf{Theorem \ref{thm4}}. The discussion of the major arc is similar to that for Theorem \ref{thm2}. We decompose $S^+(x)$ as $S(x)$ plus $O(1)$ sums of the form $U_k(E,x)$. We replace $U_k(E,x)$ by $V_k(E,x)$ with error $O(X\mc L^{-1})$. By a variant of the argument leading to \eqref{eq9.1}, we can replace
 \[\int_{-\tau}^\tau S^4(x) S^+(x) \Phi(x) e(-Rx)dx,\
 \int_{-\tau}^\tau S^3(x) S^+(x)^2 \Phi(x) e(-Rx)dx\]
respectively by
 \begin{equation}\label{eq9.12}
\sum_{(k, E)} \int_{-\tau}^\tau I^4(x) U_k(E,x) \Phi(x) e(-Rx) dx
 \end{equation}
and
 \newpage
 \begin{equation}\label{eq9.13}
\sum_{(k, E)} \ \sum_{(\ell,E')} \int_{-\tau}^\tau I^3(x) U_k(E,x) U_\ell(E',x) \Phi(x) e(-Rx) dx
 \end{equation}
with error $O(X^{5-c-c\eta} \mc L^{-6})$. We extend the integrals in \eqref{eq9.12}, \eqref{eq9.13} to $\mb R$ with error $O(X^{5-c-c\eta}\mc L^{-6})$ using Lemma \ref{lem19} (ii).

Omitting regions of summation, we have
 \begin{align*}
\int_{\mb R} &I(x)^3 U_k(E,x) U_\ell(E',x) \Phi(x) e(-Rx)dx\\[2mm]
&=\sum_{p_1,\ldots, p_{k-1}}\ \sum_{p_1',\ldots,p_{\ell-1}'} \ \frac 1{\pi_{k-1} \pi_{\ell-1}'}\ \frac 1{\mc L^3(\log X)/\pi_{k-1})(\log X/\pi_{\ell-1}')}\\[2mm]
&\int_{X_5}^X \int_{X_4}^X \int_{\frac X8}^X\int_{\frac X8}^X\int_{\frac X8}^X \int_{-\infty}^\infty e(x(t_1^c + t_2^c + t_3^c + t_4^c + t_5^c - R) \Phi(x) dt_1\ldots dt_5
 \end{align*}
where $X_5 = \max\left(\pi_{k-1}p_{k-1},\frac X8\right)$, $X_4 = \max\left(\pi_{k-1}'\, p_{k-1}', \frac X8\right)$. We write the inner integral as $\phi(t_1^c + \cdots + t_5^c - R)$ and replace $X_4$, $X_5$ by $X/8$, incurring an error that is $O(\mc L^{-6} H_5)$, by Lemma \ref{lem22} (ii). This produces the quantity
 \[\frac{H_5}{\mc L^3} \Bigg(\sum_{p_1,\ldots,p_{k-1}} \
 \frac 1{\pi_{k-1} \log \frac X{\pi_{k-1}}}\Bigg)
 \Bigg(\sum_{p_1', \ldots, p_{\ell-1}'} \
 \frac 1{\pi_{\ell-1}' \log \frac X{\pi_{\ell-1}'}}\Bigg).\]
Arguing as in the preceding proof, we arrive at representations of
 \[\int_{-\tau}^\tau S(x)^4 S^+(x) \Phi(x) e(-Rx)dx \ ,
 \int_{-\tau}^\tau S^3(x) S^+(x)^2 \Phi(x) e(-Rx)dx\]
to within a factor $1 + o(1)$, of the respective forms
 \[\frac{u^+H_5}{\mc L^5} \ , \ \frac{(u^+)^2H_5}{\mc L^5}.\]

Here, recalling \eqref{eq7.6},
 \[u^+ = 1 + d_1 + d_2;\]
the integrals
 \[d_1 = \int_{0.317}^{0.5} \ \frac{dx}{x(1-x)}\
 \text{ and }\ d_2 = \int_{0.317}^{1/3} \int_x^{\frac 12(1-x)}
 \frac{dy\, dx}{xy(1-x-y)}\]
take account of the products $p_1p_2\sim X$ $(p_1\le p_2)$ and $p_1p_2p_3 \sim X$ $(p_1 \le p_2 \le p_3)$ respectively. Simple estimations yield
 \[1 < u^+ < 1.8,\]
and Theorem \ref{thm5} follows from \eqref{eq1.6} combined with the minor arc bound of Section \ref{sec:minorarcthm4}.
 \bigskip

(vi) \textbf{Theorem \ref{thm6}}. As in the discussion of the major arc for Theorem \ref{thm5}, we replace $T(x)$, $T^+(x)$ respectively by $S(x)$, $S^+(x)$ with acceptable error. We decompose $S^+$ as $S$ plus two sums of the form $U_k(E,x)$ in \eqref{eq9.7}, $k=2,3$. Using Lemma \ref{lem22} (ii), we replace $V_k(E,x)$ by
 \begin{align*}
W_k(E,x) &= \sum_{\a_1\in E}\ \sum_{p_1\le p_2\le \cdots \le p_{k-1}} \ \frac 1{\pi_{k-1} \log \frac X{\pi_{k-1}}}\
 (v(X^c,x))\\[2mm]
 &\qquad - v\left(\max\left(\pi_{k-1}^c p_{k-1}^c, \left(\frac X8\right)^c\right)\right)
 \end{align*}
with error $O(X\mc L^{-1})$; similarly for $S(x)$. By a variant of the argument leading to \eqref{eq9.1}, we replace
 \[\int_{-\tau}^\tau S^4(x) S^+(x) e(-rx)dx \ , \
 \int_{-\tau}^\tau S^3(x) S^+(x)^3 e(-rx)dx\]
respectively by
 \begin{align}
&\sum_{(k, E)} \ \frac 1{\mc L^4} \int_{-\tau}^\tau J^4(x) W_k(E,x) e(-rx)dx,\label{eq9.14}\\[2mm]
&\sum_{(k,E)} \ \sum_{(\ell, E')} \ \frac 1{\mc L^3} \int_{\tau}^\tau J^3(x) W_k(E,x) W_\ell(E',x) e(-rx)dx\label{eq9.15} 
 \end{align}
with error $O(X^{5-c-c\eta} \mc L^{-6})$. Using Lemma \ref{lem19} (ii), with the same error we can extend the integrals in \eqref{eq9.14}, \eqref{eq9.15} to $\left[-\frac 12, \frac 12\right]$. Thus the expression in \eqref{eq9.15} has been replaced by
 \begin{align*}
&\frac 1{\mc L^3} \ \sum_{p_1,\ldots,p_k}\ \sum_{p_1',\ldots, p_{\ell-1}'} \ \frac 1{\pi_{k-1}\pi_{\ell-1}' \left(\log \, \frac X{\pi_{k-1}}\right)\left(\log\, \frac X{\pi_{\ell-1}'}\right)}\\[2mm]
&\  \int_{-\frac 12}^{\frac 12} J^3(x)(v(X^c, x)-v(X_4^c,x))(v(X^c,x) - v(X_5^c,x))e(-rx)dx,
 \end{align*}
where $X_4 = \max\left(\pi_{k-1}\pi_{k-1},\frac X8\right)$ $X_5 = \max\left(\pi_{\ell-1}' p_{\ell-1}',\frac X8\right)$. We replace $X_4$,$X_5$ by $\frac X8$, incurring an error that is $O(\mc L^{-6}L_5)$. This produces the quantity
 \[\frac{L_5}{\mc L^3} \ \sum_{p_1,\ldots, p_{k-1}} \
 \frac 1{\pi_{k-1}\left(\log\, \frac X{\pi_{k-1}}\right)} \
 \sum_{p_1', \ldots, p_{\ell-1}'} \ \frac 1{
 \pi_{\ell-1}' \left(\log\, \frac X{\pi_{\ell-1}'}\right)}.\]
We can now complete the proof of Theorem \ref{thm6} in a similar manner to that of Theorem \ref{thm4}, with $L_5$ in place of $H_5$, since $S^+$ is the same as in Theorem \ref{thm4} and we have
 \[L_5 \gg r^{\frac 5c - 1}\]
by Lemma \ref{lem20}.

 \end{document}